\documentclass[a4paper,11pt]{amsart}
\usepackage{amsfonts, amssymb, amsgen, amsbsy, amstext, amsopn, amsmath,
latexsym, indentfirst,color,longtable,verbatim, accents}
\usepackage{dsfont}
\usepackage[english]{babel}
\usepackage[utf8]{inputenc}
\usepackage{microtype}

\usepackage[colorlinks=true, allcolors=black]{hyperref}

\usepackage[backend=bibtex, firstinits=true, doi=false, isbn=false, url=false, style=numeric, maxbibnames=99]{biblatex}
\renewbibmacro{in:}{}
\bibliography{biblio-rough_levelsets.bib}

\usepackage{verbatim} 
\renewcommand{\H}{\mathbb{H}}

\newcommand{\E}{\mathbb{E}}

\newcommand{\G}{\mathbb{G}}

\newcommand{\M}{\mathbb{M}}

\newcommand{\R}{\mathbb{R}}

\newcommand{\V}{\mathbb{V}}

\newcommand{\Z}{\mathbb{Z}}

\newcommand{\cL}{\mathcal{L}}

\newcommand{\cS}{\mathcal{S}}

\newcommand{\cb}{\mathfrak{b}}

\newcommand{\vth}{\vartheta}

\newcommand{\lls}{\big(}
\newcommand{\rrs}{\big)}

\newcommand{\diam}{\mbox{\rm diam}}

\newcommand{\supp}{\mbox{\rm supp}}

\newcommand{\res}{\mbox{\LARGE{$\llcorner$}}}

\newcommand{\der}{\partial}

\newcommand{\beqas}{\begin{eqnarray*}}
\newcommand{\eeqas}{\end{eqnarray*}}
\newcommand{\beqa}{\begin{eqnarray}}
\newcommand{\eeqa}{\end{eqnarray}}
\newcommand{\beq}{\begin{equation}}
\newcommand{\eeq}{\end{equation}}
\newcommand{\bce}{\begin{center}}
\newcommand{\ece}{\end{center}}

\newcommand{\pa}[1]{\left( #1 \right)}               
\newcommand{\set}[1]{\left\{ #1 \right\}}            

\newcommand{\qandq}{\quad\mbox{and}\quad}

\newtheorem{The}{Theorem}[section]
\newtheorem{Lem}[The]{Lemma}
\newtheorem{Def}[The]{Definition}
\newtheorem{Rem}[The]{Remark}
\newtheorem{Pro}[The]{Proposition}
\newtheorem{Cor}[The]{Corollary}
\newtheorem{Exa}[The]{Example}

\newcommand{\bt}{\begin{The}}
\newcommand{\et}{\end{The}}
\newcommand{\bl}{\begin{Lem}}
\newcommand{\el}{\end{Lem}}
\newcommand{\bd}{\begin{Def}\rm}
\newcommand{\ed}{\end{Def}}
\newcommand{\br}{\begin{Rem}\rm}
\newcommand{\er}{\end{Rem}}
\newcommand{\bpr}{\begin{Pro}}
\newcommand{\epr}{\end{Pro}}
\newcommand{\bc}{\begin{Cor}}
\newcommand{\ec}{\end{Cor}}
\newcommand{\bex}{\begin{Exa}}
\newcommand{\eex}{\end{Exa}}

\theoremstyle{plain}
\newtheorem{theorem}{Theorem}[section]
\newtheorem{proposition}[theorem]{Proposition}
\newtheorem{lemma}[theorem]{Lemma}
\newtheorem{corollary}[theorem]{Corollary}

\newtheorem{definition}[theorem]{Definition}

\theoremstyle{remark}
\newtheorem{remark}[theorem]{Remark}

\let\altphi\phi
\let\phi\varphi
\let\varphi\altphi
\let\altphi\undefined
\newcommand{\veps}{\varepsilon}

\newcommand{\bra}[1]{\left(#1\right)}
\newcommand{\cur}[1]{\left\{#1\right\}}
\newcommand{\sqa}[1]{\left[#1\right]}

\newcommand{\abs}[1]{\left\lvert#1\right\rvert}
\newcommand{\norm}[1]{\left\lVert#1\right\rVert}

\renewcommand{\O}{\H}

\let\d\undefined
\newcommand{\d}{\mathop{}\!\mathrm{d}}

\newcommand{\dist}{\mathsf{d}}

\newcommand{\h}{\mathsf{h}}
\renewcommand{\v}{\mathsf{v}}
\newcommand{\phd}{\,\cdot\,}
\newcommand{\gh}[1]{\left[ #1 \right]^\h}
\newcommand{\gv}[1]{\left[ #1 \right]^\v}
\newcommand{\openball}[2]{B_{#1}\! \left( #2 \right)}
\newcommand{\ob}[2]{\bar{\mathrm{B}}_{#1}\! \left( #2 \right)}
\renewcommand{\cb}[2]{\mathbb{B}_{#1}\! \left( #2 \right)}

\renewcommand{\c}{\mathop{}\!\mathrm{c}}
\newcommand{\A}{U}
\newcommand{\nalp}[1]{\alpha, #1}

\newcommand{\conealp}[2]{C^{1,\alpha}_\h(#1, #2)}
\newcommand{\conetwo}{C^{1,\alpha}_\h(\H, \R^2)}
\newcommand{\nablah}{\nabla_{\h}}
\newcommand{\rt}{\mathrm{R}}
\newcommand{\lsd}{LSDE }
\renewcommand{\E}{\mathrm{E}}
\newcommand{\gA}{\mathrm{A}}
\renewcommand{\V}{\mathrm{K}}
\renewcommand{\chi}{\mathds{1}}

\title{A rough calculus approach to level sets in the Heisenberg group}

\begin{document}

\author{Valentino Magnani}
\address{Valentino Magnani, Dipartimento di Matematica, Universit\`a di Pisa \\
Largo Bruno Pontecorvo 5 \\ I-56127, Pisa}
\email{valentino.magnani@unipi.it}

\author{Eugene Stepanov}
\address{St.Petersburg Branch of the Steklov Mathematical Institute of the Russian Academy of Sciences,
Fontanka 27,
191023 St.Petersburg, Russia
\and
Department of Mathematical Physics, Faculty of Mathematics and Mechanics,
St. Petersburg State University, Universitetskij pr.~28, Old Peterhof,
198504 St.Petersburg, Russia
\and ITMO University}
\email{stepanov.eugene@gmail.com}

\author{Dario Trevisan}
\address{Dario Trevisan, Dipartimento di Matematica, Universit\`a di Pisa \\
Largo Bruno Pontecorvo 5 \\ I-56127, Pisa}
\email{dario.trevisan@unipi.it}

\date{\today}

\thanks{This work is supported by the University of Pisa, 
Project\ PRA\_2016\_41.
The second author also acknowledges the support of the St.~Petersburg State University grants \#6.60.1355.2016 and
\#6.38.223.2014, 
the Russian government grant \#074-U01 and
the RFBR grant \#14-01-00534. }
\subjclass[2010]{Primary 53C17, 28A75. Secondary 26A42.}
\keywords{Heisenberg group; rough paths; implicit function theorem;  coarea formula.} 
\date{\today}

\begin{abstract}
We introduce novel equations, in the spirit of rough path theory, that pa\-ram\-etrize level sets of intrinsically regular maps on the Heisenberg group with values in $\R^2$. These equations can be seen as a sub-Riemannian counterpart to classical ODEs arising from the implicit function theorem. We show that they enjoy all the natural well-posedness properties, thus allowing for a ``good calculus'' on nonsmooth level sets. We apply these results to prove an area formula for the intrinsic measure of level sets, along with the corresponding coarea formula.
\end{abstract}

\maketitle


\section{Introduction}
The classical implicit function theorem asserts that regular level sets of a $C^1$ smooth map 
on a Euclidean space are $C^1$ smooth, with a natural 
parametrization which can be written in terms of differential equations
involving the first derivatives of the map. 
This is no longer the case for maps on
Carnot-Carath\'{e}odory spaces which are regular only with respect to the
intrinsic geometry of their domain. 
In this paper, we study the simplest of such situations, 
where maps are defined on the first Heisenberg group $\H$ 
and take values in $\R^2$, so that their level sets are expected to be one-dimensional. Indeed, our results show that these level sets can be still 
represented by curves, that are in general only H\"{o}lder continuous, and not anymore smooth, 
but still solve a peculiar analogue of a classical ODE, 
that we call \emph{Level Set Differential Equation} (LSDE). 

At first glance, the LSDE is similar to equations driven by a rough signal 
appearing in the theory of \emph{rough paths} as exposed e.g.\ in~\cite{friz_hairer_rough_2014, gubinelli_controlling_2004, lyons_differential_1998}, 
but it is  different, being inherently ``autonomous'', 
while the usual \emph{rough differential equations} (RDEs) are not.  
However, the theory of rough paths still provides an 
appropriate tool, namely, the
\emph{sewing lemma}, to construct solutions to LSDEs, thus enabling some ``differential'' calculus 
for maps on the Heisenberg group, 
regular only in the intrinsic sense of the latter, but possibly nowhere differentiable on a set of positive measure~\cite{magnani_coarea_scalar_2005}.
Before presenting our results, it is worth describing
the mathematical landscape motivating this study.


\subsection*{Intrinsically regular level sets.}

A natural problem of Geometric Measure Theory in graded nilpotent Lie groups is the study of the structure of level sets of maps intrinsically ($\h$-)differentiable in the sense of P.\ Pansu \cite{pansu_metriques_1989}, which are known to be quite complicated \cite{vershik_1987, sussmann_1992, agrachev_1997}, in particular not even  rectifiable in the classical sense \cite{kirchheim_serra-cassano_level_sets_2004} and
neither can be interpreted as metric currents \cite{ambrosio_kirchheim_currents_2000, ambrosio_kirchheim_rectifiable_2000}. For $\h$-differentiable maps $F:\G\to\M$ between graded nilpotent Lie groups,
a convenient parametrization for level sets is available when a {\em semidirect factorization} of $\G$ with respect to $\M$ holds, as a consequence of a suitable implicit function theorem \cite{magnani_towards_2013}. 
In case $\G = \H^n$, the $n$-th Heisenberg group, topologically identified with $\R^{2n+1}$, and $\M =\R^k$,  the respective semidirect factorization is known to exist precisely when $1\le k\le n$. If that is the case, level sets of $F$ can be seen as naturally acting on 
intrinsic differential forms of the Heisenberg group, forming the
so-called Rumin's complex \cite{rumin_complex_1990}. As a result,
they become intrinsic Heisenberg currents
and their intrinsic measure can be computed by a suitable area formula \cite{fssc_currents_2007},  leading to a  coarea formula  \cite{magnani_area_2011}.

When  $n<k\le 2n$ there is no general approach to the structure of level sets, which then cannot even be seen as Heisenberg currents. 
The simplest model of these difficult cases is $n=1$ and $k=2$. 
Here, the fact that level sets at regular points are continuous curves and cannot degenerate to a singleton
was established by the first author and G.P.\ Leonardi \cite{leonardi_intersections_2011}, using an {\em ad-hoc} 
method exploiting classical ODEs. Following a purely geometric approach that relies on a Reifenberg-type flatness estimate, A.~Kozhevnikov~\cite{kozhevnikov_roughness_2011, kozhevnikov_proprietes_2015} showed that they are in fact H\"older continuous curves, furnishing in this way also  a coarea formula for a large subclass of $\h$-differentiable maps from $\H$ to $\R^2$.


\subsection*{An Euclidean view of the problem} 
In case $F\colon \R^3\to \R^2$, $x = (x^1, x^2, x^3)$, $F(x)= (F_1(x), F_2(x))$ is a smooth ($C^1$) map, 
the differential of $F$ at $x \in \R^3$ with respect to the ``horizontal coordinates'' $x^1$, $x^2$ is represented by the square matrix
\[
\nabla_{12} F (x) := \left(
\begin{array}{cc}
  \partial_{1} F_1(x) & \partial_{2} F_1(x) \\
  \partial_{1} F_2 (x) & \partial_{2} F_2(x)
\end{array}
\right).
\]
If $p \in \R^3$ is nondegenerate with $\nabla_{12} F(p)$ invertible, then the implicit function theorem implies that 
the level set $F^{-1}(F(p))$ can be parametrized, locally around $p$, by a $C^1$ curve 
$\gamma: I \to \R^3$, $t \mapsto \gamma_t$, with the parameter $t$, on the interval $I$. After a one dimensional change of variable, we may consider $\gamma^3$ as a local ``vertical coordinate'', i.e.\ $\gamma^3_t:=p^3+t$, that is $\dot{\gamma}^3_t=1$, for $t \in I$. Moreover, the ``horizontal components'' $(\gamma^1$, $\gamma^2)$ solve an ODE, as a consequence of the condition $\frac{d}{dt} F(\gamma_t) = 0$, namely 
\begin{equation}
\label{eq:ode-intro}
\left( 
\begin{array}{c}
 \dot{\gamma}^1_t \\
  \dot{\gamma}^2_t \\
\end{array}
\right) =-
\bra{\nabla_{12} F (\gamma_t)}^{-1} 
\left( 
\begin{array}{c}
 \partial_{3} F_1(\gamma_t) \\
 \partial_{3} F_2(\gamma_t)
\end{array}
\right) \quad \text{for $t \in I$.}
\end{equation}

\subsection*{The nonholonomic problem}
The situation radically changes in the ``nonholonomic'' setting, where 
 ``horizontal'' directions are represented by a couple of smooth vector fields $X_1$, $X_2$,
such that the classical  {\em Lie bracket generating condition}
\beq\label{eq:LieBGenerating}
\mathrm{span}\set{X_1(x), X_2(x), [X_1, X_2](x)}=\R^3 \quad \text{at every $x \in \R^3$}
\eeq 
holds, 
$[\cdot,\cdot]$ standing for the Lie bracket, and the ``horizontal plane'' at $x\in \R^3$ is
defined by $\mathrm{span}\set{X_1(x), X_2(x)}$. 
The model situation is that of $F$ being defined on the first Heisenberg group $\mathbb{H}$,  topologically identified with $\R^3$,
and equipped with ``horizontal vector fields''
\[
 X_1(x^1, x^2, x^3):= \der_1 - x^2\der_3\qandq   X_2(x^1, x^2, x^3):= \der_2+ x^1\der_3.
\]
Let $F\colon \mathbb{H}\to \R^2$ be  continuously horizontally ($\h$-)differentiable.  
Even if the ``horizontal differential'' of $F$, represented by the square matrix 
\[
\nabla_\h F := \left(
\begin{array}{cc}
  X_1 F_1 & X_2 F_1 \\
  X_1 F_2 & X_2 F_2
\end{array}
\right),
\]
is invertible at some point $p \in \H$ (i.e., $p$ is nondegenerate) the loss of Euclidean regularity may allow for highly irregular level sets in a neighbourhood of $p$. If $F$ is $C^1$ in the Euclidean sense, then one could rewrite~\eqref{eq:ode-intro}
in terms of $\nabla_\h F$ instead of $\nabla_{12}F$, getting
\begin{eqnarray} 
\label{eq:ode-intro-h}
&\left( 
\begin{array}{c}
 \dot{\gamma}^1_t \\
  \dot{\gamma}^2_t \\
\end{array}
\right) =-
\bra{\nabla_{\h} F (\gamma_t)}^{-1} 
\left( 
\begin{array}{c}
 \partial_{3} F_1(\gamma_t) \\
 \partial_{3} F_2(\gamma_t)
\end{array}
\right) \theta_{\gamma_t}(\dot \gamma_t), \quad \text{where}\\
&\theta= dx^3 + x^2 dx^1- x^1 dx^2 \nonumber
\end{eqnarray}
is the contact form of $\H$. It is then natural to consider 
\beq\label{eq:theta-introduction}
\theta_{\gamma_t}(\dot \gamma_t) = \dot \gamma_t^3 + \gamma_t^2 \dot \gamma_t^1- \gamma_t^1 \dot \gamma_t^2=1 \quad \text{for $t \in I$}
\eeq
as the condition that replaces $\dot \gamma_t^3 = 1$ occurring in the Euclidean case, closing the system~\eqref{eq:ode-intro-h}.
Here the main difficulty appears when $F$ is only $\h$-differentiable, since in this case
the ``vertical derivatives'' $\partial_3F^1$ and $\partial_3F^2$ may not exist, so 
the system~\eqref{eq:ode-intro-h} makes no sense.
In addition, the $1/2$-H\"older continuity of the sub-Riemannian distance with respect to the Euclidean one leads us to a genuinely H\"older continuous curve $\gamma$ (hence possibly nowhere differentiable), which makes even the definition of the term $\gamma_t^2 \dot \gamma_t^1- \gamma_t^1 \dot \gamma_t^2$ in~\eqref{eq:theta-introduction} a nonsense. 
From a geometric viewpoint, the inherent obstacle to this approach arises from the 
fact that $\gamma$ cannot move along horizontal directions
and in this case there is no suitable geometric notion of differentiability.
If $\dot\gamma_t$ were horizontal, i.e.\ $t \mapsto \gamma_t$ were differentiable at $t \in I$ in the sense of \cite{pansu_metriques_1989}, then the chain rule would give
\[0 =  \frac{d}{dt} F(\gamma_t) = \nablah F(\gamma_t) \left(\begin{array}{c}
 \dot{\gamma}^1_t \\
 \dot{\gamma}^2_t \end{array}\right)\]
and from the non-degeneracy of $\nablah F(\gamma_t)$ we would get $\dot{\gamma}^1_t = \dot{\gamma}^2_t =0$, in conflict with the natural injectivity requirements.
In other words, the parametrization of the level set $\gamma$ at any point must move along ``vertical directions''.


\subsection*{Description of results}
In this paper, we will prove that the analogy with the 
Euclidean situation 
can be suitably extended to the nonholonomic case. 
If the horizontal gradient $\nablah F$ is $\alpha$-H\"{o}lder continuous with respect to the sub-Riemannian distance of $\H$ for some $\alpha\in(0,1)$,
we are able to provide a rigorous counterpart of~\eqref{eq:ode-intro}, which is no longer an ODE but rather a new ``differential equation', the LSDE (Definition~\ref{d:LSDE}). In fact, instead of derivatives, we consider finite differences of the 
unknown solution $\gamma$ at ``sufficiently close'' 
points $s$, $t \in I$. For instance, the third component $\gamma^3$ of the solution is requested to satisfy
\beq\label{eq:contact_discrete}
(\gamma_s^{-1}\gamma_t)^3=t-s + o(|t-s|)
\eeq
for $|t-s|$ small, with the appropriate order of ``error'' $o(\cdot)$.
This seems to be the natural translation of ~\eqref{eq:theta-introduction} into our framework, where the inverse and the product are given by the group operation.
Such use of finite differences allows us to circumvent the
regularity problems of the na\"ive approach: the terms $\gamma_t^2 \dot \gamma_t^1- \gamma_t^1 \dot \gamma_t^2$ in~\eqref{eq:theta-introduction} are replaced by
$\gamma_t^2( \gamma_{s}^1 - \gamma_{t}^1) - \gamma_t^1( \gamma_{s}^2 - \gamma_{t}^2)$,
 and the partial derivatives $\partial_3F^1$, $\partial_3F^2$ in~\eqref{eq:ode-intro-h} are replaced  by an expression involving  the 
remainder of the  ``horizontal'' Taylor expansion of $F$.
To ``integrate'' a consistent family of such local descriptions,
we use an important technical tool underlying the theory of rough/controlled paths, the so-called \emph{sewing lemma}, which in this case leads to integrals in the sense of L.~C.~Young~\cite{young_inequality_1936}.
As already mentioned, the LSDE is an autonomous equation and does not fit precisely in the framework of RDEs, but we may imagine that the ``noise'' is self-induced by~\eqref{eq:contact_discrete}.

Our main result (Theorem~\ref{thm:parametrization}) is a version of the implicit function theorem, showing that any level set of $F$ in a neighbourhood of a nondegenerate point $p\in \H$ can  be parametrized by an injective 
continuous curve $\gamma$ satisfying an \lsd.  Further properties of solutions to LSDEs are proven, such as uniqueness (Theorem~\ref{thm:loc_uniqueness_I}) and stability with respect to approximations of $F$ (Corollary~\ref{cor:stability}).
As two applications, we provide an area formula for level sets (Theorem~\ref{thm:area}, Corollary~\ref{cor:area-lsde}) as well as a coarea formula (Theorem~\ref{thm:coarea}) for maps with H\"older continuous horizontal gradient, where the \lsd is instrumental to follow the approach of \cite{magnani_area_2011}.

\subsection*{Acknowledgements} 
We are grateful to Massimiliano Gubinelli for an important remark on a preliminary version of the proof of Theorem~\ref{thm:existence}. We thank all the speakers of the workshop ``Singular Phenomena and Singular Geometries'', Pisa, 20-23 June 2016, for fruitful discussions, which much contributed to inspire the present work.

\section{Preliminaries}

\subsection*{General notation}
Throughout, we use the notation $\abs{\phd}$ for the Euclidean norm of a vector. For $\beta \in [0,1]$, given an interval $I$ and a function $f: I \to \R^k$, ($k \ge 1$), $t \mapsto f_t$, we let
\[ \norm{f}_{\beta} := \sup_{\substack{s,t \in  I\\ s \neq t}} \frac{ \abs{f_t - f_s}}{\abs{t-s}^\beta} \in [0, \infty],\]
and write $f \in C^\beta(I, \R^k)$ if  $\norm{f}_{\beta}< \infty$. Notice that $\norm{f}_0 \le 2\sup_{t \in I} \abs{f_t}$.

Next, we  introduce notation and basic results on the geometry and analysis of maps in the Heisenberg group. We follow throughout the exposition \cite{folland_hardy_1982}, other approaches can be found, e.g.\ in \cite{burago_course_2001, gromov_metric_1999}.

\subsection*{Heisenberg group} We represent the Heisenberg group $\H$ as $\R^3$ equipped with the non-commutative group operation $(x,y) \to xy$  defined by
\begin{equation}\label{eq:heisenberg-operation}(x^1, x^2, x^3) (y^1, y^2, y^3) =  (x^1+y^1, x^2+y^2, x^3+y^3 +(x^1y^2-x^2y^1)) \end{equation}
where $x =(x^1, x^2, x^3)$ and  $y = (y^1, y^2, y^3)$. 

 In what follows, for simplicity of notation,  we let $x^\h = (x^1, x^2) \in \R^2$ (respectively, $x^\v = x^3 \in \R$) denote the ``horizontal'' (respectively ``vertical'') components of $x = (x^1, x^2, x^3) \in \H$, so that we identify $x =(x^\h, x^\v)$. We notice that $x \mapsto  x^\h$ is a group homomorphism, since~\eqref{eq:heisenberg-operation} gives  $(xy)^\h = x^\h + y^\h$, while $x \mapsto x^\v$ is not.

\subsection*{Dilations and gauges} For $r\ge 0$, we let $\delta_r: \H \to \H$ denote the intrinsic dilation (which is a group homomorphism)
\[ x \mapsto \delta_r(x) = ( r x^\h, r^2 x^\v) = (rx^1,rx^2,r^2x^3).\]
Clearly, $\delta_r \circ \delta_{r'} = \delta_{r r '}$, for any $r$, $r' \ge 0$. 

It is useful to introduce the following non-negative functions on $\H$: the ``horizontal gauge'', $\gh{\phd}$, and the ``vertical gauge'', $\gv{\phd}$,
\[ \gh{x}   := | x^\h| = \sqrt{ \abs{x^1}^2 + \abs{x^2}^2}, \quad \quad \gv{x}  := \sqrt{\abs{x^\v}} = \sqrt{ \abs{x^3}}.\]
Both $[ \phd ]^\h$ and  $[ \phd ]^\v$ are $1$-homogeneous with respect to dilations, i.e.,
\[  \gh{\delta_r x} = r \gh{x}, \quad  \gv{\delta_r x} = r \gv{x}, \quad \text{for $x\in \H$, $r\ge 0$.}\]

\subsection*{Invariant distances}

We fix from now on a distance $\dist:\H\times\H\to\R$, which is \emph{left-invariant} with respect to the group operation and $1$-homogeneous with respect to dilations, i.e.,
\[
\dist(x,y)=\dist(zx,zy)\qandq \dist(\delta_rx,\delta_ry)=r \, \dist(x,y),
\]
for $x,y,z\in\H$ and $r\ge0$. Closed (respectively, open) balls of center $x \in H$ and radius $r\ge 0$ are denoted by $\cb{r}{x}$ (respectively, $\ob{r}{x}$).
A fundamental example of such a distance is the so-called Carnot-Carath\'eodory distance associated to a left-invariant horizontal distribution of vector fields.

By $1$-homogeneity, e.g.~as in \cite[Proposition~1.5]{folland_hardy_1982}, there exists some constant $\c = \c(\dist)\ge 1$ such that
\begin{equation}\label{eq:equivalence-distance}
 \c^{-1}\pa{ \gh{x^{-1}y}  + \gv{x^{-1}y} } \le \dist (x,y) \le \c \pa{ \gh{x^{-1}y}  + \gv{x^{-1}y} }
\quad \text{for $x, y \in \H$.}
\end{equation}

\subsection*{Horizontally differentiable maps}

We introduce the following left invariant vector fields on $\H$, seen as derivations,
\[ X_1(x^1, x^2, x^3)= \der_1 - x^2\der_3, \quad  X_2(x^1, x^2, x^3)= \der_2+ x^1\der_3,\]
where $\partial_i$ denotes the derivation along the direction $x^i$, $i \in \cur{1, 2,3}$.
Their linear span at any point $x \in \H$ defines the the so-called {\em horizontal distribution}, which is well-known to be totally non-integrable. The Carnot-Carath\'eodory distance associated to $X_1$, $X_2$ yields the {so-called} sub-Riemannian distance  on $\H$.

For $k\ge1$, $\alpha \in (0,1]$, given a function $g: \O \to \R^k$ and a subset $\A \subseteq  \H$ we let
\[ \norm{g}_{\nalp{\A}} := \sup_{\substack{x,y\in\A\\ x\neq y }} \frac{\abs{g(x)- g(y)}} {\dist(x,y)^\alpha}.\]

For $F: \O \to \R^k$, we write $F \in \conealp{\O}{\R^k}$ if both derivatives $X_1 F (x)$, $X_2 F(x)$ exist at every  $x \in \O$ and the \emph{horizontal Jacobian matrix}  $\nablah F(x) := \bra{ X_1 F(x), X_2F(x) }$ satisfies
\[  \norm{ \nablah F}_{\alpha, \A}< \infty, \quad \text{for every bounded $\A \subseteq \H$.}\]
We say that a sequence $(F^n)_{n\ge1} \subseteq \conealp{\O}{\R^k}$ converge to $F \in \conealp{\O}{\R^k}$ if, for every bounded set $\A \subseteq \H$, we have $F^n \to F$ uniformly in $\A$ and $\norm{ \nablah F -\nablah F^n }_{\alpha, \A}\to 0$ as $n \to \infty$.

\begin{remark}
It would be more appropriate  (but heavier) to use the notation $C^{1,\alpha}_{\h, {loc}}(\O, \R^k)$ instead of $\conealp{\O}{\R^k}$, because maps $F \in\conealp{\O}{\R^k}$ have only \emph{locally} H\"older continuous derivatives.
\end{remark}

\begin{definition}
A function $F: \H \to \R^k$ is called {\em $\h$-differentiable} at $x \in \H$, if there exists a group homomorphism $\d_{\h}  F(x): \H \to \R^k$ such that
\[  \lim_{y \to x} \frac{ \abs{ F(y) - F(x) - \d_{\h}  F(x) \bra{ x^{-1}y} }} {\dist(x,y)} = 0.\]
We say that $x \in \H$ is \emph{nondegenerate} for $F$ if $\d_{\h}  F(x)$ is injective (when $k=2$, this amounts to $\d_{\h}  F(x)$ invertible).
\end{definition}

In all what follows, for $x$, $y \in \H$, we write
\begin{equation}\label{eq:taylor-definition} \rt(x,y) :=  F(y) - F(x) - \d_{\h}  F(x) \bra{ x^{-1}y}\end{equation}
for the first-order horizontal Taylor expansion in $x$, evaluated at $y$. The implicit dependence upon $F$ in such notation will be always clear from the context.

If $F\in \conealp{\O}{\R^k}$, the stratified Taylor inequality \cite[Theorem~1.42]{folland_hardy_1982} ensures that $F$ is $\h$-differentiable at every $x \in \O$, with
\begin{equation}\label{eq:d-nabla} \d_{\h}  F(x) (y) = \nablah   F(x) \, y^\h = X_1 F(x) \, y^1 + X_2 F(x)\,  y^2.\end{equation}
The same result guarantees that there exists some constant $\c = \c(\dist) \ge 1$ such that
\begin{equation}\label{eq:taylor-two-points}
\abs{ \rt(x,y) } \le \c \norm{\nablah   F}_{\alpha, \ob{\c r}{x}} \dist(x,y)^{1+\alpha}, \quad \text{for any $x$, $y \in \H$, with $\dist(x,y) \le r$.}
\end{equation}
Let us stress the fact that $\norm{\nablah   F}_{\alpha, \ob{\c r}{x}}$  above is on the ball of center $x$ and radius $\c \,  r$.

For technical reasons, it will be useful to use horizontal Taylor expansions at a fixed point $p \in \O$, relying on the algebraic identity 
\begin{equation}\label{eq:taylor-three-points}  \rt(p,y)-\rt(p,x) =  F(y) - F(x) - \d_{\h}  F (p) (x^{-1}y ) \quad \text{for $x$, $y \in \O$,}\end{equation}
following from~\eqref{eq:taylor-definition} and the fact that $\d_{\h}  F(p)$ is a homomorphism.

For $x,y \in \ob{r}{p}$, $r \ge 0$, one has
%
\begin{equation*}\begin{split}\abs{\rt(p,y)-\rt(p,x)} & \le \abs{ F(y) - F(x) - \d_{\h}  F (x) (x^{-1}y )} + \abs{  \bra{ \d_h F (x)- \d_h F(p)} (x^{-1}y )} \\
& = \abs{ \rt(x,y)} + \abs{  \bra{ \nablah F (x)- \nablah F(p)} (x^{-1}y )^\h}\quad \text{by~\eqref{eq:taylor-definition},~\eqref{eq:d-nabla}}  \\
&\le \c \norm{\nablah  F}_{\alpha, \ob{2 \! \c r }{p}}  \bra{\dist(x,y)^{1+\alpha} +  \dist(p,x)^\alpha \gh{x^{-1} y}  }\quad \text{by~\eqref{eq:taylor-two-points}}\\
&\le \c \norm{\nablah  F}_{\alpha, \ob{2 \! \c r }{p}}  \bra{\dist(x,y)^{1+\alpha} +  r^\alpha \gh{x^{-1} y}  }\\
&\le \c \norm{\nablah  F}_{\alpha, \ob{2 \! \c r }{p}}  \bra{\c^{1+\alpha} \bra{\gh{x^{-1}y} + \gv{x^{-1}y} }^{1+\alpha} +  r^\alpha \gh{x^{-1} y}  },\\
\end{split}\end{equation*}
the latter inequality coming from~\eqref{eq:equivalence-distance}. Applying the inequality $\abs{a + b}^{1+\alpha} \le  2^\alpha \bra{a^{1+\alpha}+ b^{1+\alpha}}$ with $a = \gh{x^{-1}y}$, $b = \gv{x^{-1}y}$ we get, for some constant $\c = \c(\dist, \alpha) \ge 1$,
\begin{equation}\label{eq:error-taylor-three-points}
\abs{\rt(p,y)-\rt(p,x)}  \le \c \norm{\nablah  F}_{\alpha, \ob{2 \! \c r }{p}} \bra{ r^\alpha \gh{x^{-1}y} + \bra{\gv{x^{-1}y} }^{1+\alpha}}\, \text{for $x$, $y \in \ob{r}{p}$,} \end{equation}
where we also used~\eqref{eq:equivalence-distance} to estimate the term
\[ \bra{\gh{x^{-1}y}}^{1+\alpha} \le  \bra{\c\, \dist(x,y)}^{\alpha} \gh{x^{-1}y} \le \bra{2\c^2}^\alpha  r^{\alpha} \gh{x^{-1}y}.\]
We also mention the weaker version of~\eqref{eq:error-taylor-three-points}, which follows from it arguing as above with $\gv{x^{-1}y}$ in place of $\gh{x^{-1}y}$: 
\begin{equation}\label{eq:error-taylor-three-points-weak}
\abs{\rt(p,y)-\rt(p,x)}  \le \c \norm{\nablah  F}_{\alpha, \ob{2 \! \c r }{p}} r^\alpha \bra{ \gh{x^{-1}y} + {\gv{x^{-1}y} }}\, \text{for $x$, $y \in \ob{r}{p}$,} \end{equation}
for some $\c = \c(\dist, \alpha) \ge 1$.
%

\section{The level set differential equation}

We introduce our main objects of study, i.e., suitable ``differential equations'' which provide parametrizations of level sets of a function $F\in \conetwo$, in a neighbourhood of a nondegenerate point $p \in \O$, i.e.\ when the matrix $\nablah F(p)$ invertible. In what follows, we fix  $\alpha \in (0,1]$.

\begin{definition}[Level set differential equation]\label{d:LSDE} Let $p \in \H$ be a nondegenerate point for $F\in \conetwo$. Given an interval $I \subseteq \R$, we say that $\gamma: I \to \H$, $t \mapsto \gamma_t$, is a solution to the level set differential equation (LSDE) if $\gamma$ is continuous and  
    \renewcommand{\arraystretch}{1.2}
\begin{equation}\label{eq:yde-group}
  \left\{ \begin{array}{ll}  \bra{ \gamma^{-1}_s\gamma_t}^\h & =  -\nablah  F(p)^{-1} \bra{  \rt(p,\gamma_t)-\rt(p,\gamma_s)} \\
  \bra{ \gamma^{-1}_s\gamma_t}^\v  & =  t-s + \E_{st} \end{array} \right. \quad
 \quad \text{for every $s, t \in I$,}
\end{equation}
with $\E_{st}:I^2 \to \R^2$ satisfying
\begin{equation}\label{eq:error}
\norm{\E} := \sup_{\substack{s,t\in I \\ s \neq t }} \frac{\abs{ \E_{st} }}{\abs{t-s}^{1+\alpha}} < \infty.
\end{equation}
\end{definition}

\begin{remark}[Concentration on level sets] \label{rem:concentration} Any solution $\gamma$ to the \lsd is concentrated on a level set of $F$, i.e.\ $t \mapsto F(\gamma_t)$ is constant. Actually, this follows from the ``horizontal'' (i.e.\ first) equation in~\eqref{eq:yde-group} only, for
\begin{equation*}
\begin{split}
F(\gamma_t) - F(\gamma_s) & = \d_{\h}  F(p) \bra{ \gamma_{s}^{-1} \gamma_{t}} + R(p, \gamma_{t})- R(p, \gamma_{s}) \\
& = \nablah  F(p) \bra{ \bra{\gamma_{s}^{-1} \gamma_{t}}^\h + \nablah  F(p)^{-1} \bra{R(p, \gamma_{t})- R(p, \gamma_s)} }= 0.
\end{split}\end{equation*}
\end{remark}

\begin{remark}[On the vertical equation]
Using the group operation~\eqref{eq:heisenberg-operation}, we can rewrite the ``vertical'' (i.e.\ second) equation  in~\eqref{eq:yde-group} as
\begin{equation*} 
 \gamma_t^\v-\gamma_s^\v = t-s + (\gamma^1_s \gamma^2_t - \gamma^1_t \gamma^2_s) + \E_{st}, \quad \text{for  $s, t \in I$.}\end{equation*}
\end{remark}

\begin{remark}[On ``errors'']
The term $\E_{st}$  should be regarded as a natural ``error'' arising from the fact that the \lsd is in fact a difference equation, rather than a differential one, in the spirit of controlled equations as developed e.g.\ in \cite{gubinelli_controlling_2004}. Condition~\eqref{eq:error}  becomes crucial to make sure that the contributions of $\E_{st}$ are infinitesimal, in some sense. In principle, one could allow as well for an error term $\E_{st}^\h$ also in the ``horizontal'' equation in~\eqref{eq:yde-group} but then, under an assumption on $\E^\h_{st}$ similar to~\eqref{eq:error}, one can prove that necessarily  $\E_{st}^\h =0$, hence we directly formulate the \lsd without such a term.
\end{remark}

\begin{remark}[On the role of $p$]
Taking into account~\eqref{eq:taylor-three-points}, one could also think of replacing $p$ with $\gamma_s$, possibly allowing for an additional error $\E_{st}^\h$. However, for our technique to work, namely for the sewing lemma (Lemma \ref{lem:sewing} below) to be applicable, such a choice seems to restrict the validity of our arguments only to $\alpha > 1/2$ (if one reiterates all the computations of this paper). Moreover, let us notice that we are not requiring $\gamma_t = p$ for some $t$: actually, we let $I = [-\delta, \delta]$, for some $\delta>0$, and choose $\gamma_0$ sufficiently close (but not necessarily equal) to $p$.
\end{remark}

%
%

We end this section with a basic result showing that the term $t-s$ in the ``vertical'' equation in~\eqref{eq:yde-group} prevents $\gamma$ from being constant, actually forcing its local injectivity.

\begin{lemma}[Local injectivity]\label{lem:injective}
Let $I \subseteq \R$ be an interval, $\gamma: I \to \H$ be a solution to the LSDE associated to $F$, with $p \in \H$ nondegenerate. Then, there exist $\delta>0$ and $\varrho>0$ such that
\begin{equation}\label{eq:local-injectivity}
\abs{t-s}^{1/2} \le \varrho\,  \dist(\gamma_s, \gamma_t), \quad \text{for $s$, $t \in I$, $\abs{t-s} \le 2 \delta$.}
\end{equation}
\end{lemma}

\begin{proof}
The ``vertical'' equation in~\eqref{eq:yde-group} gives
\[  t-s  = (\gamma_s^{-1} \gamma_t)^\v - \E_{st}, \]
hence, if $s$, $t \in I$ satisfy $\abs{t-s} \le 2\delta$, then
\begin{equation}\label{eq:vertical-injectivity-proof}\abs{t-s}  \le \abs{ (\gamma_s^{-1} \gamma_t)^\v} + \abs{\E_{st}} 
  \le \bra{\c\, \dist(\gamma_s, \gamma_t)}^2 + (2 \delta)^\alpha \norm{\E} \abs{t-s},\end{equation}
 by~\eqref{eq:equivalence-distance} and~\eqref{eq:error}. If we choose $\delta>0$, $\varrho>0$ such that
 \begin{equation}\label{eq:conditions-injective-lemma} (2\delta)^\alpha \norm{\E} \le \frac 1 2 \quad \text{and} \quad \varrho^2 \ge 2\c^2,  \end{equation}
we obtain, from~\eqref{eq:vertical-injectivity-proof},
\[\abs{t-s} \le  2 \bra{\c \, \dist(\gamma_s, \gamma_t)}^2 =  2 \c^2 \dist(\gamma_s, \gamma_t)^2\le  \bra{\varrho\dist(\gamma_s, \gamma_t)}^2,\]
hence the thesis.
\end{proof}

\section{Existence of solutions}

To provide existence of some solution to the \lsd we rely on the fundamental tool of the theory of controlled paths, sometimes called \emph{sewing lemma}, which allows us to cast the differential equations into an ``integral'' form, and perform a Schauder fixed point argument. 

\begin{lemma}[Sewing lemma]\label{lem:sewing}
For $\alpha \in (0,1]$, $k \ge 1$, there exists some constant $\kappa > 0$ such that the following holds. For any interval $I$ and continuous $\gA : I^2 \to \R^k$ that satisfies
\[ \abs{ \gA_{st} - \gA_{su} - \gA_{ut} } \le \norm{\gA} \abs{t-s}^{1+\alpha}, \quad \text{for $s$, $u$, $t \in I$ with $s\le u \le t$,}\]
for some constant $\norm{\gA}$, then there exists a continuous function $f: I \to \R$  such that
\begin{equation}\label{eq:sewing} \abs{f_t - f_s - \gA_{st} } \le \kappa \norm{\gA} \abs{t-s}^{1+\alpha} , \quad \text{for $s$,  $t \in I$.}\end{equation}
\end{lemma}

For a proof, we refer e.g.\ to \cite[Lemma~2.1]{feyel_curvilinear_2006} (see also \cite[Theorem~2]{feyel_non-commutative_2008} for more general moduli of continuity). 

\begin{remark}[Young integrals]\label{rem:young-integral}
The theory of integration in the sense of Young, introduced in the seminal paper \cite{young_inequality_1936}, can be recovered as an instance of the sewing lemma. Actually, the \lsd could be stated as well as an integral Young equation, but we chose to adopt the modern point of  view as in \cite{gubinelli_controlling_2004}. Indeed, for $g^1: I \to \R$, $g^2: I \to \R$, define
\begin{equation}\label{eq:germ-young} \gA_{st} := g^1_s \bra{ g^2_t - g^2_s}, \quad \text{$s$, $t \in I$.}\end{equation}
Then, for $s$, $t$, $u \in I$, with $s\le u \le t$,
\begin{equation}\label{eq:delta-young}\begin{split} \gA_{st} - \gA_{su} - \gA_{ut} & = g^1_s \bra{ g^2_t - g^2_s}-g^1_s \bra{ g^2_u - g^2_s}-g^1_u \bra{ g^2_t - g^2_u} \\
& = \bra{ g^1_s -g^1_u} \bra{ g^2_t - g^2_u}.
\end{split}\end{equation}
Therefore, if $g^1 \in C^{\beta_1}(I; \R)$, $g^2\in  C^{\beta_2}(I; \R)$, with $\beta_1+\beta_2 >1$,
\begin{equation} \label{eq:norm-young} \abs{ \gA_{st} - \gA_{su} - \gA_{ut}} \le \norm{g^1}_{\beta_1} \norm{g^2}_{\beta_2} \abs{t-s}^{\beta_1+\beta_2},\end{equation}
and the sewing lemma applies, yielding a function $f$, which one could show \cite{gubinelli_controlling_2004} that satisfies $f_t  = f_0 + \int_0^t g^1_s dg^2_s$, where integration is in the sense of Young.
\end{remark}

\begin{remark}[Uniqueness]\label{rem:sewing_unique}
Clearly, if $f$ satisfies~\eqref{eq:sewing} and we add to $f$ a constant function, the sum still satisfies~\eqref{eq:sewing}, hence we may always additionally prescribe the value of $f$ at some (but only one) $t \in I$. Moreover, we have uniqueness up to additive constants, in the following sense: if $g:I\to\R$ satisfies
\[
\limsup_{\substack{s,t \in I \\ \abs{s-t}\to 0}} \frac{|g_t-g_s-\gA_{st}|}{\abs{t-s}} = 0,
\]
then $f-g$ is a constant function.
\end{remark}

\begin{theorem}[Existence]\label{thm:existence}
Let $p \in \H$ be a nondegenerate point for $F\in \conetwo$. Then, there exist positive $\delta_0$, $\veps_0$, $\varrho_0$ such that, for any $q \in \ob{\veps_0}{p}$, there is an injective solution $\gamma$ to the \lsd on the interval $[-\delta_0, \delta_0]$, with $\gamma_0=q$,
\begin{equation}\label{eq:holder-horizontal} 
  \|\gamma^\h \|_{\frac{1+\alpha}{2}} \le  \varrho_0 \quad \text{and} \quad \norm{\E} \le \kappa  \|\gamma^\h \|_{\frac{1+\alpha}{2}}^2 \le \kappa  \varrho_0^2.
\end{equation}
\end{theorem}

The proof relies on an application of the Schauder fixed point theorem, i.e.\ we find a convex invariant set $\V$ for a map $\Phi$ naturally defined by the LSDE. The main technical difficulty, however, is in showing that $\Phi$ is continuous in an appropriate topology: here we deal with an integral, defined implicitly by the sewing lemma, which does not allow us to pass the absolute value inside, as the in the case of Lebesgue integral, which makes the argument  much more involved than the standard one working for ODE's.

\begin{proof}
 For simplicity,  write in what follows positive $\delta$, $\veps$, $\varrho$ to be chosen sufficiently small, yielding $\delta_0$, $\veps_0$ and $\varrho_0$ as in the thesis. Write $I := [-\delta, \delta]$ and fix $q \in \ob{\veps}{p}$. 

{ \emph{Introduction of space and map.}} We introduce the following compact, convex subset $\V$ of $C^{\frac{1+\alpha}{2}}(I; \R^2)$, 
\[ \V = \cur{ \eta = (\eta^1, \eta^2): I \to \R^2 \, \Big| \, \eta_0 = q^\h,\ \norm{\eta}_{\frac{1+\alpha}{2}} \le  \varrho }.
 \]
We define the map $\Phi$, on $\V$, $\eta \mapsto \Phi(\eta)$, in two steps. First, for $\eta \in \V$, we apply Lemma~\ref{lem:sewing} with
\beq\label{eq:A3st}
\gA_{st} := t-s - (\eta^1_t \eta^2_s - \eta^1_s \eta^2_t) \quad \text{for $s, t \in I$}
\eeq
to obtain a (unique) function $f: I \to \R$, such that $f_0 = q^\v$. Then, we define
\begin{equation}\label{eq:defbar-eta} \bar{\eta}_t := (\eta_t, f_t) \in  \H, \quad \text{for $t \in I$,}\end{equation}
and finally set
\begin{equation}\label{eq:definition-phi} \Phi(\eta)_t := q^\h + \nablah  F(p)^{-1} \bra{  R(p, \bar\eta_t)-R(p,q)},\quad\text{ for $t \in I$.}\end{equation}

{ \emph{The map $\Phi$ is well defined:}} it suffices to show that the sewing lemma can be applied to~\eqref{eq:A3st}. Adding and subtracting the quantity $\eta^2_s\eta^1_s$, we rewrite the right hand side in~\eqref{eq:A3st} as
\begin{equation}\label{eq:a3st-young} \gA_{st}  = t-s - \eta^2_s (\eta^1_t -\eta^1_s) + \eta^1_s (\eta^2_t-\eta^2_s),\end{equation}
which shows that $\gA$ is a sum of $t-s$ and two other terms of Young type, i.e.\ as in~\eqref{eq:germ-young}. Therefore, arguing as in~\eqref{eq:delta-young}, we obtain
\[
\gA_{st} - \gA_{su} - \gA_{ut} = -\bra{\eta^2_s - \eta^2_u}\bra{\eta^1_t - \eta^1_u} + \bra{\eta^1_s - \eta^1_u}\bra{\eta^2_t - \eta^2_u},
\]
hence, as in~\eqref{eq:norm-young}, for $s\le u \le t$,
\[ \abs{ \gA_{st} - \gA_{su} - \gA_{ut} } \le \norm{\eta}_{\frac{1+\alpha}{2}}^2 \abs{t-s}^{1+\alpha}
\]
Thus, we are in a position to apply Lemma~\ref{lem:sewing},  obtaining $f:I\to\R$ with $f_0 = q^\v$ and
\begin{equation}\label{eq:error-eta3-sewing-lemma} \abs{ f_t - f_s - \gA_{st} } \le  \kappa \norm{\eta}_{\frac{1+\alpha}{2}}^2 \abs{t-s}^{1+\alpha} \quad \text{for $s, t \in I$.}
\end{equation}

{ \emph{Claim: $\bar \eta_t \in \ob{2\veps}{p}$.}} We provide conditions on $\delta$, $\veps$, $\varrho$ which ensures the claim. By definition~\eqref{eq:defbar-eta} of $\bar{\eta}$ and~\eqref{eq:A3st},
\begin{equation*}
 (\bar{\eta}_s^{-1} \bar{\eta}_t)^\v = \bar{\eta}^\v_t - \bar{\eta}^\v_s + (\eta^1_t \eta^2_s - \eta^1_s \eta^2_t) = f_t - f_s - \gA_{st} + t-s,\end{equation*}
hence~\eqref{eq:error-eta3-sewing-lemma} and the conditions $\abs{t-s} \le 2\delta$, $\norm{\eta}_{\frac{1+\alpha}{2}} \le \varrho$ imply
\[ \abs{  (\bar{\eta}_s^{-1} \bar{\eta}_t)^\v} \le \kappa \varrho^2 \abs{t-s}^{1+\alpha}+ \abs{t-s}   \le \bra{1+ \kappa \varrho ^2 (2\delta)^\alpha } \abs{t-s}\]
so that
\begin{equation}\label{eq:estim_vertical-tilde_eta}
\gv{ \bar{\eta}_s^{-1} \bar{\eta}_t} \le \sqrt{1+ \kappa \varrho^2  (2\delta)^\alpha } \abs{t-s}^{1/2}.
\end{equation}
Since
\begin{equation} \label{eq:estim_horizontal-tilde_eta} \gh{\bar{\eta}_s^{-1} \bar{\eta}_t }   = \abs{ \bra{\eta_t^1 -\eta_s^1, \eta^2_t - \eta^2_s}} \le \norm{\eta}_{\frac{1+\alpha}{2}} \abs{t-s}^{\frac{1+\alpha}{2}} \le \varrho \bra{2\delta}^{\alpha/2}  \abs{t-s}^{1/2},
\end{equation}
the bound between $\dist$ and the sum of the horizontal and vertical gauges~\eqref{eq:equivalence-distance} yields
\begin{equation*}
\dist(\bar{\eta}_s, \bar{\eta}_t)  \le \c \bra{ \varrho (2\delta)^{\alpha/2} + \sqrt{1+ \kappa \varrho^2  (2\delta)^\alpha }}  \abs{t-s}^{1/2}.
\end{equation*}
If $\varrho$ and $\delta$ satisfy
\begin{equation}\label{eq:condition-existence-1}
\varrho (2\delta)^{\alpha/2} \le 1,
\end{equation}
it follows that
\[ \dist(\bar{\eta}_t, \bar{\eta}_s) \le  \c (1+\sqrt{1+\kappa}) \abs{ t-s}^{1/2} \quad \text{for $s$, $t \in I$.}\]
For $s=0$, $\bar \eta_0 = q$, so that $\dist(\bar{\eta}_t, q)  \le  \c (1+\sqrt{1+\kappa}) \delta^{1/2}$. If $\delta$ and $\veps$ satisfy
\begin{equation}\label{eq:condition-existence-2}
 \c (1+\sqrt{1+\kappa}) \delta^{1/2} \le \veps,
\end{equation}
then $\dist(\bar{\eta}_t, p) \le \dist(\bar{\eta}_t, q) + \dist(q,p)$ gives $\bar{\eta}_t \in \ob{2\veps}{p}$, for $t \in I$, i.e.\ the claim.

{ \emph{$\Phi$ maps $\V$ into itself.}} We give further conditions on $\delta$, $\veps$, $\varrho$ to ensure that $\Phi(\V) \subseteq \V$. Since~\eqref{eq:definition-phi} for $t =0$ gives $\Phi(\eta)_0 = q^\h$,  we only have to prove that $\norm{ \Phi (\eta)}_{\frac{1+\alpha}{2}} \le \varrho$.  
%
To this aim, we estimate
\begin{equation*}\begin{split}
\| \Phi(\eta)_t &- \Phi(\eta)_s\|   = \abs{\nablah  F(p)^{-1} \bra{ \rt(p,\bar{\eta}_t)-\rt(p,\bar{\eta}_s)} } \quad \text{by~\eqref{eq:definition-phi}} \\
& \le \abs{\nablah  F(p)^{-1}} \abs{ \rt(p,\bar{\eta}_t)-\rt(p,\bar{\eta}_s)} \\
& \le \c  \abs{\nabla_{\h}  F(p)^{-1} }  \norm{\nablah  F}_{\alpha, \ob{4 \! \c \veps }{p}} \bra{ (2\veps)^\alpha \gh{\eta_s^{-1}\eta_t} + \bra{\gv{\eta_s^{-1}\eta_t} }^{1+\alpha} }\\
& \qquad \qquad \text{by the previous claim and~\eqref{eq:error-taylor-three-points} with  $x =\bar \eta_s$, $y =\bar \eta_t$, $r = 2 \veps$} \\
 & \le \c  \abs{\nabla_{\h}  F(p)^{-1} } \norm{\nablah  F}_{\alpha, \ob{4 \! \c \veps }{p}} \bra{ (2\veps)^\alpha  \norm{\eta}_{\frac {1+\alpha}{2}} + \bra{1+ \kappa \varrho^2  (2\delta)^\alpha }^{\frac {1+\alpha}{2}} }  \abs{t-s}^{\frac{1+\alpha}{2}} \\
 & \qquad \qquad \text{by~\eqref{eq:estim_horizontal-tilde_eta} and~\eqref{eq:estim_vertical-tilde_eta}}\\
 & \le \c  \abs{\nabla_{\h}  F(p)^{-1} } \norm{\nablah  F}_{\alpha, \ob{4 \! \c \veps }{p}} \sqa{ (2\veps)^\alpha  \varrho + \bra{1+ \kappa}^{\frac {1+\alpha}{2} }}  \abs{t-s}^{\frac{1+\alpha}{2}} \quad \text{by~\eqref{eq:condition-existence-1},}
 \end{split} \end{equation*}

We conclude that, if $\delta$, $\veps$ and $\varrho$, satisfy~\eqref{eq:condition-existence-1},~\eqref{eq:condition-existence-2} and
\begin{equation}\label{eq:condition-existence-3}
 \c  \abs{\nabla_{\h}  F(p)^{-1} } \norm{\nablah  F}_{\alpha, \ob{4 \! \c \veps }{p}} \bra{ (2\veps)^\alpha  \varrho + \bra{1+ \kappa}^{\frac {1+\alpha}{2} }} \le \varrho,
\end{equation}
then $\Phi(\V) \subseteq \V$. Let us then fix $\delta=\delta_0$, $\veps=\veps_0$ and $\varrho= \varrho_0>0$ such that these conditions are satisfied: this can be achieved e.g.\ choosing first $\veps_0>0$ and $\varrho_0$ such that~\eqref{eq:condition-existence-3} holds and then choosing $\delta_0>0$ small enough so that both~\eqref{eq:condition-existence-1} and ~\eqref{eq:condition-existence-2} holds. To ensure that any solution (to be obtained) $\gamma$ be injective, we also require that~\eqref{eq:conditions-injective-lemma} in Lemma~\ref{lem:injective} hold with $\delta= \delta_0$ and with $\varrho = \varrho_0$, with $\kappa \varrho^2$ in place of $\norm{E}$. 
%
%

{\emph{Existence and properties of fixed points.}} Let us fix $\beta \in (0,\alpha)$. Taking for granted continuity of $\Phi: \V \to \V$ in the topology of $C^{\frac{1+\beta}{2}}(I; \R^2)$, which will be proven in a further technical step, by compactness of the embedding $C^{\frac{1+\alpha}2}(I; \R^2)$ into $C^{\frac{1+\beta}{2}}(I; \R^2)$ is compact, we apply Schauder fixed point theorem, see e.g.\ \cite[Theorem 11.1]{gilbarg_elliptic_2001}, obtaining some $\eta \in \V$ such that $\Phi(\eta) = \eta$. 

Let us show that  $\gamma= \bar \eta$  defined by~\eqref{eq:defbar-eta} solves the \lsd with $\gamma_0 = q$ and~\eqref{eq:holder-horizontal}. Indeed, by definition of $\bar{\eta}$, we have immediately that $\bar{\eta}_0 = q$ and the first inequality in~\eqref{eq:holder-horizontal} holds because $\eta \in \V$. The ``horizontal'' equation in~\eqref{eq:yde-group} holds for $s$, $t \in I$, because
\[ \bra{\bar{\eta}_s^{-1}\bar \eta _t}^\h = \eta_t - \eta _s = \Phi(\eta)_t - \Phi(\eta)_s = - \nablah F(p)^{-1} \bra{\rt(y,\bar{\eta}_t) - \rt(y,\bar{\eta}_s)},\]
by ~\eqref{eq:definition-phi}. 
For the ``vertical'' equation in ~\eqref{eq:yde-group}, we notice that, by~\eqref{eq:A3st},
\[ \E_{st} = (\gamma_s^{-1}\gamma_t)^\v - (t-s) = \eta^\v_t - \eta^\v_s - \gA^3_{st},\]
hence the second inequality in~\eqref{eq:holder-horizontal} follows from~\eqref{eq:error-eta3-sewing-lemma}, recalling that $f = \eta^\v$ therein. Finally, $\gamma$ is injective on $I$ by Lemma~\ref{lem:injective} in view of the choice of $\delta_0$ and $\varrho_0$.

{\emph{Continuity of $\Phi$.}}   We see that conditions~\eqref{eq:condition-existence-1},~\eqref{eq:condition-existence-2} and~\eqref{eq:condition-existence-3} imply that $\Phi: \V \to \V$ is continuous with respect to the topology induced by the norm $\norm{\cdot}_{\frac{1+\beta}{2}}$, for any $\beta \in (0,\alpha)$.
The argument, relying on a real interpolation, is close to that in \cite[Proposition~5]{gubinelli_controlling_2004}. For $\eta$, $\zeta \in \V$ and $t \in I$, by~\eqref{eq:definition-phi}, we have
\begin{equation*}\begin{split}
\abs{ \Phi(\eta)_t - \Phi(\zeta)_t}&  = \abs{\nablah  F(p)^{-1} \bra{  \rt(p,\bar{\eta}_t)-\rt(p,\bar{\zeta}_t)} } \\
& \le \c\abs{\nablah  F(p)^{-1} }\norm{\nabla_{\h}  F}_{\alpha, \ob{4 \! \c \veps }{p} } (2\veps)^\alpha \bra{  \gh{\bar \zeta_t^{-1}\bar \eta_t} +  \gv{\bar \zeta_t^{-1} \bar \eta_t}   }, \\
\end{split}\end{equation*}
the inequality following  from~\eqref{eq:error-taylor-three-points-weak}, applied with $x = \bar \zeta_t$, $y= \bar \eta_t$ (recall that $\bar{\eta}_t$, $\bar{\zeta}_t \in \ob{2\veps}{p}$).

Denote $\xi_t := \eta_t - \zeta_t$. Since $\zeta_0 = \eta_0 = q^\h$, we have $\xi_0 = 0$ and 
\begin{equation}\label{eq:norm-xi} \gh{\bar \zeta_t^{-1}\bar \eta_t} = \abs{\xi_t} \le \norm{ \xi}_0 \le \delta^{\frac{1+\beta}{2}} \norm{\xi}_{\frac {1+\beta}{2}}.\end{equation}
To estimate the term $\gv{\bar \zeta_t^{-1} \bar \eta_t}$, we notice first that the group operation~\eqref{eq:heisenberg-operation} yields
\[ (\bar \zeta_t ^{-1} \bar \eta_t)^\v = (\bar{\eta}_t^\v - \bar{\zeta}_t^\v) + (\eta^1_t \zeta^2_t - \zeta^1_t \eta^2_t). \]
The second term in the above sum is easily estimated, since
\begin{equation*}
\begin{split}
\abs{ \eta^1_t \zeta^2_t - \zeta^1_t \eta^2_t } & = \abs{ \bra{\eta^1_t -\zeta^1_t} \zeta^2_t - \zeta^1_t \bra{ \eta^2_t- \zeta^2_t}}  = \abs{ \xi^1_t \zeta^2_t -\zeta^1 \xi^2_t }
\le \abs{\zeta_t} \abs{\xi_t}\\
 &\le  \bra{|q^\h|+\varrho \delta^{\frac{1+\alpha}{2}}}  \delta^{\frac{1+\beta}{2}} \norm{\xi}_{\frac {1+\beta}{2}},
\end{split}
\end{equation*}
using the last inequality of~\eqref{eq:norm-xi} and
\begin{equation}
\label{eq:norm-zeta}
\sup_{t \in I} \abs{ \zeta_t}  \le |{q^\h}| + \norm{\zeta}_{\frac{1+\alpha}{2}} t^{\frac{1+\alpha}{2}} \le |q^\h|+\varrho \delta^{\frac{1+\alpha}{2}}.
\end{equation}
 Hence, we are reduced to find a bound on  $\bar\eta_t^\v- \bar\zeta_t^\v$. Below, we prove that for some constant $\c>0$ independent of $\eta$ and $\zeta$, one has
\begin{equation} \label{eq:bound-vertical-difficult}
\abs{  \bar \eta^\v_t- \bar \zeta^\v_t } \le \c \norm{ \xi}_{\frac{1+\beta}{2}}.
\end{equation}
Once~\eqref{eq:bound-vertical-difficult} is proved, we conclude that 
\[\norm{ \Phi(\eta) - \Phi(\zeta) }_{0} =  \sup_{ t \in I} \abs{ \Phi(\eta)_t - \Phi(\zeta)_t} \le \c \norm{ \xi}_{\frac{1+\beta}{2}},\]
for some (different)  $\c>0$ independent of $\eta$ and $\zeta$.  Using the bound
\[
\norm{ \Phi(\eta) - \Phi(\zeta) }_{\frac{1+\alpha}{2}} \le \norm{ \Phi(\eta) }_{\frac{1+\alpha}{2}} + \norm{\Phi(\zeta) }_{\frac{1+\alpha}{2}} \le 2 \varrho,
\]
together with the interpolation inequality
\[\norm{\cdot}_{\frac{1+\beta}{2}} \le  \norm{\cdot}_0^{\frac{\alpha-\beta}{1+\alpha}} \norm{\cdot}_{\frac{1+\alpha}{2}}^{\frac{1+\beta}{1+\alpha}}\]
finally gives, again for some (different) $\c >0$,
\[ \norm{ \Phi(\eta) - \Phi(\zeta) }_{\frac{1+\beta}{2}} \le \c \norm{ \xi}_{\frac{1+\beta}{2}}^{\frac{\alpha - \beta}{1+\alpha}}  = \c \norm{ \eta - \zeta}_{\frac{1+\beta}{2}}^{\frac{\alpha - \beta}{1+\alpha}},\]
which yields continuity of $\Phi$.

\emph{Proof of~\eqref{eq:bound-vertical-difficult}.}  This follows from another application of the sewing lemma, but using its uniqueness part. Indeed, we denote $\gA_{st}(\eta)$ (resp.\ $\gA_{st}(\zeta)$) the right hand side of~\eqref{eq:A3st} (resp.\ with $\zeta$ instead of $\eta$) and define (only) here 
\[ \gA_{st} := \gA_{st}(\eta) - \gA_{st}(\zeta).\]

Since $\bar \eta^\v$ and $\bar \zeta^\v$ are both built via sewing lemma, we  have
\[\frac{ \abs{ (\bar\eta_t^\v - \bar\zeta_t^\v)- (\bar\eta_s^\v - \bar\zeta_s^\v) - \gA_{st}}}{\abs{t-s}} \le \frac{\abs{ \bar \eta_t ^\v- \bar \eta_s ^\v- \gA_{st}(\eta)}}{\abs{t-s}} +  \frac{ \abs{ \bar \zeta_t ^\v- \bar \zeta_s ^\v- \gA_{st}(\zeta)}}{\abs{t-s}} \to 0\]
as $\abs{t-s} \to 0$. Next, we check that the sewing lemma applies to $\gA_{st}$, but looking for a quantitative bound in terms of $\xi$. Uniqueness will imply the required estimate. To this aim, recalling that $\gA_{st}(\eta)$, $\gA_{st}(\zeta)$ can be rewritten as in~\eqref{eq:a3st-young}, recollecting all the terms that appear, we see that $\gA_{st}$ is given by the difference between
\begin{equation} \label{eq:first-term-continuity} \zeta^2_s \bra{\zeta^1_t -\zeta^1_s}  -  \eta^2_s  \bra{ \eta^1_t -\eta^1_s}\end{equation}
and an analogous term,
\begin{equation}\label{eq:second-term-continuity}  \zeta^1_s \bra{ \zeta^2_t-\zeta^2_s}- \eta^1_s \bra{ \eta^2_t-\eta^2_s}.\end{equation}
Adding and subtracting the quantity $\bra{\zeta^1_t -\zeta^1_s} \eta^2_s$ in~\eqref{eq:first-term-continuity}, we transform it into
\begin{equation}\label{eq:first-term-continuity-clever} 
 - \xi_s^2 \bra{\zeta^1_t -\zeta^1_s} + \eta^2_s\bra{ \xi^1_t - \xi^1_s }.\end{equation}
The absolute value of the latter expression can be bounded from above by
\begin{equation*}\begin{split} \sup_{u\in I} \abs{\xi_u} & \norm{\zeta}_{\frac {1+\alpha}{2}} \abs{t-s}^{\frac{1+\alpha}{2}}  + \sup_{u\in I}\abs{\eta_u} \norm{\xi}_{\frac {1+\beta}{2}} \abs{t-s}^{\frac{1+\beta}{2}} \le  \\
& \le  
  \bra{  \varrho  (2 \delta)^{\frac{1+\alpha}{2}}  + \bra{ |q^\h|+\varrho \delta^{\frac{1+\alpha}{2}}}    }(2\delta)^{\frac{1 + \beta}{2}}\norm{\xi}_{\frac{1+\beta}{2}} \quad \text{by~\eqref{eq:norm-xi},~\eqref{eq:norm-zeta} with $\eta$ instead of $\zeta$} \\
 & \le  
  \bra{ |q^\h|+2 \varrho  (2 \delta)^{\frac{1+\alpha}{2}}}(2\delta)^{\frac{1 + \beta}{2}}\norm{\xi}_{\frac{1+\beta}{2}}.  \\ 
   \end{split}
 \end{equation*}
Arguing with~\eqref{eq:second-term-continuity} in a similar way,  we conclude that for $s$, $t \in I$,
\begin{equation}\label{eq:bound-ga-continuity} \abs{ \gA_{st} } \le 2\bra{ |q^\h| + 2 \varrho (2\delta)^{\frac{1+\alpha}{2}} }  (2\delta)^{\frac{1+ \beta}{2}}\norm{\xi}_{\frac{1+\beta}{2}}.\end{equation}

Now we estimate $\gA_{st} - \gA_{su}-\gA_{ut}$,  for $s$, $t$, $u \in I$,  with $s\le u \le t$. Arguing as in~\eqref{eq:germ-young},~\eqref{eq:delta-young} of Remark~\ref{rem:young-integral}, we see that~\eqref{eq:first-term-continuity-clever} yields a contribution to this quantity equal to
\[-\bra{\xi_s^2- \xi_u^2} \bra{\zeta^1_t -\zeta^1_u}  +\bra{\eta^2_s - \eta^2_u} \bra{ \xi^1_t - \xi^1_u },\]
the absolute value of which is estimated from above by
\[ \bra{ \norm{\zeta}_{\frac{1+\alpha}{2}} +  \norm{\eta}_{\frac{1+\alpha}{2}}} \norm{\xi}_{\frac{1+\beta}{2}} \abs{t-s}^{1 + \bra{\alpha+ \beta}/2} \le 2 \varrho \norm{\xi}_{\frac{1+\beta}{2}} \abs{t-s}^{1 + \bra{\alpha+ \beta}/2}.  \]
If we argue with the contribution of~\eqref{eq:second-term-continuity} in a similar way, we conclude that 
\[  \abs{\gA_{st} - \gA_{su}-\gA_{ut}} \le 4\rho \norm{\xi}_{\frac{1+\beta}{2}} \abs{t-s}^{1 + \bra{\alpha+ \beta}/2}.\]

Thus, we are in a position to apply Lemma~\ref{lem:sewing}, obtaining a (unique) function $f: I \to \R$ such that $f_0 =0$ and
\begin{equation}\label{eq:sewing-continuity}\abs{ f_t - f_s - \gA_{st}} \le 4 \kappa \rho \norm{\xi}_{\frac{1+\beta}{2}} \abs{t-s}^{1 + \bra{\alpha+ \beta}/2}, \quad \text{for $s$, $t \in I$.}\end{equation}
with $\kappa =\kappa(\bra{\alpha+ \beta}/2)$. From Remark~\ref{rem:sewing_unique} it follows that $f = \bar \eta^\v- \bar \zeta^\v$. By~\eqref{eq:bound-ga-continuity} and~\eqref{eq:sewing-continuity}, with $s=0$, $t \in I$, we conclude that
\[ \abs{  \bar \eta^\v_t- \bar \zeta^\v_t } \le \sqa{ 4 \kappa \varrho \delta^{1+\bra{\alpha+\beta}/{2}}  +   2\bra{ |q^\h| +  2(2\delta)^{\frac{1+\alpha}{2}} \varrho }} (2\delta)^{\frac{1+ \beta}{2}}\norm{\xi}_{\frac{1+\beta}{2}}, \]
as claimed.
\end{proof}

The proof of Theorem~\ref{thm:existence} yields the following sensitivity result. Actually, one could provide an alternative proof of Theorem~\ref{thm:existence} by proving first the following result and then approximating $F$ in $C^{1,\alpha}_\h(\O, \R^2)$ with a sequence of smooth functions $\cur{F^n}_{n\ge 1}$.

\begin{corollary}[Sensitivity]\label{cor:compactness}
 Let $p \in \H$ be a nondegenerate point for $F\in \conetwo$ and let $\cur{F^n}_{n\ge 1} \subseteq \conetwo$ converge to $F$ in $\conetwo$. Then, there exist $\bar n \ge 1$ and positive $\delta_0$, $\veps_0$, $\varrho_0$ such that, for any $n \ge\bar n$ and $q^n \in \ob{\veps_0}{p}$, there is an injective solution $\gamma^n:[-\delta_0, \delta_0] \to \H$  to the \lsd associated to $F^n$, with $\gamma_0^n = q^n$ and
\begin{equation}\label{eq:holder-horizontal-stability} 
  \|\bra{\gamma^n}^\h \|_{\frac{1+\alpha}{2}} \le  \varrho_0 \quad \text{and} \quad \norm{\E^n} \le \kappa  \|\bra{\gamma^n}^\h \|_{\frac{1+\alpha}{2}}^2.
\end{equation}
Moreover, the family $(\gamma^n)_{n\ge1}$ is compact (with respect to uniform convergence) and any limit point is an injective solution $\gamma$ to the \lsd associated to $F$, such that~\eqref{eq:holder-horizontal} holds.
\end{corollary}

\begin{proof}
First, we notice that the constant(s) $\c$ appearing in~\eqref{eq:condition-existence-2} and~\eqref{eq:condition-existence-3} are independent of $F$. As $n \to \infty$, since $\norm{\nablah F^n - \nablah F}_{\alpha, \A} \to 0$ for every bounded $U \subseteq \H$, choosing $\A = \ob{4 \! \c }{p}$, for $n$ large enough, $\nablah F^{n}(p)^{-1}$ exists and 
\[ \nablah F^n(p)^{-1} \to \nablah F^n(p)^{-1} \quad \text{and} \quad \norm{\nablah F^n}_{\alpha, \ob{4 \! \c }{p}} \to \norm{\nablah F}_{\alpha, \ob{4 \! \c }{p}}.\]
Therefore, if we choose $\delta_0$, $\veps_0$, $\varrho_0$ such that~\eqref{eq:condition-existence-1},~\eqref{eq:condition-existence-2} and~\eqref{eq:condition-existence-3} hold true for $F$ as strict inequalities, as well as~\eqref{eq:conditions-injective-lemma}, with $\kappa \varrho_0^2$ instead of  $\norm{E}$, then there exists $\bar n \ge 1$ such that for any $n \ge \bar n$, strict inequalities hold in the analogues of~\eqref{eq:condition-existence-1},~\eqref{eq:condition-existence-2} and~\eqref{eq:condition-existence-3}, as well as~\eqref{eq:conditions-injective-lemma}, with $F^n$ instead of $F$.

Then, the arguments in the proof of Theorem~\ref{thm:existence} apply for $F^n$ and provide existence of some injective solution $\gamma^n: [-\delta_0, \delta_0] \to \H$  of the \lsd associated to $F^n$, with $\gamma^n_0 = q^n \in \ob{\veps_0}{p}$. Moreover, the analogues of~\eqref{eq:holder-horizontal} with $F^n$ and $\E^n$ (instead of $F$ and $\E$)  hold uniformly in $n \ge \bar n$. This yields compactness for $\gamma^n: [-\delta_0, \delta_0] \to \H$, with respect to uniform convergence. Indeed, given $s$, $t \in [-\delta_0, \delta_0]$ we have $\gamma^n_t \in \ob{4 \! \c \veps_0 }{p}$ and, by~\eqref{eq:equivalence-distance},
\begin{equation*}
\begin{split} \dist(\gamma^n_s, \gamma^n_t) &\le \c \bra{ \gh{(\gamma^n_s)^{-1}\gamma^n_t} + \gv{(\gamma^n_s)^{-1}\gamma^n_t}} \\
& \le \c \bra{ \varrho_0 \abs{t-s}^{\frac{1+\alpha}{2}}+ \sqrt{ \abs{t-s}+ \kappa \varrho_0^2 \abs{t-s}^{1+\alpha}} },
\end{split}\end{equation*}
using the  ``vertical'' equation and the uniform bound on $\norm{\E^n}$.

Finally, to show that any limit point $\gamma$ of $\cur{\gamma^n}_{n \ge 1}$ solves the \lsd associated to $F$, we recall that $\nablah F^n(p)^{-1} \to \nablah F^n(p)^{-1}$ and notice that $\rt^n(p,x)$,  defined by~\eqref{eq:taylor-definition} with $F^n$ in place of $F$ converge to $\rt(p,x)$ uniformly in $\ob{4 \! \c \veps_0}{p}$, hence both equations in~\eqref{eq:yde-group} pass to the limit along any converging subsequence $\cur{\gamma^{n_k}}_k$.  Moreover,~\eqref{eq:error} immediately yields that $\E \mapsto \norm{\E}$ is lower semicontinuous with respect to uniform convergence of $\gamma$, hence in the limit $\norm{\E} \le \kappa \varrho_0^2$. Finally, $\gamma$ is injective on $[-\delta_0, \delta_0]$ because of Lemma~\ref{lem:injective} and our choice of $\delta_0$ and $\varrho_0$, so that~\eqref{eq:conditions-injective-lemma} holds.
\end{proof}

\section{Parametrization of level sets}


In this section, we prove that any solution to the \lsd provides a local parametrization of the level set of $F$, where it is concentrated, see Remark~\ref{rem:concentration}. The argument relies on the following two lemmas. 

\begin{lemma}[``Horizontal'' injectivity]\label{lem:uniqueness1}
Let $p \in \H$ be a nondegenerate point for $F\in \conetwo$. Then, there exists $\veps_1>0$ such that, whenever
$x,y \in \ob{\veps_1}{p}$ satisfy
\[ F(x) = F(y) \quad \text{and} \quad \gv{ x^{-1}y } \le  \gh{ x^{-1} y},  \]
we must have $x =y$. In particular, the second condition holds  if $[x^{-1}y]^\v =0$.
\end{lemma}

\begin{proof}
For simplicity, write throughout the proof $\veps>0$ to be specified below, yielding $\veps_1$ as in the thesis. Using the condition $F(x) = F(y)$ in~\eqref{eq:taylor-three-points}, we obtain
\[
 \bra{x^{-1}y }^\h = -\nablah  F(p)^{-1} \bra{\rt(p,y)-\rt(p,x)}.
\]
From~\eqref{eq:error-taylor-three-points-weak}, with $\veps$ instead of $r$, we have
\[\gh{x^{-1}y } \le\c \abs{\nablah  F(p)^{-1}} \norm{ \nablah F}_{\alpha, \ob{2\! \c \veps}{p} } \veps^\alpha \bra{ \gh{x^{-1}y} +{\gv{x^{-1}y} }}.\]
If $\veps_1=\veps>0$ is chosen so that
\begin{equation}\label{eq:condition-injectivity} \c \abs{\nablah  F(p)^{-1}} \norm{ \nablah F}_{\alpha, \ob{2\! \c \veps}{p} } \veps^\alpha < \frac 1 2,\end{equation}
then
\[\gh{x^{-1}y } < \frac 1 2  \bra{ \gh{x^{-1}y} +{\gv{x^{-1}y} }} \le \gh{x^{-1}y}\]
using $\gv{ x^{-1}y } \le  \gh{ x^{-1} y}$, which implies $\gh{x^{-1}y}=0$ hence also $\gv{x^{-1}y}=0$, i.e.\ $x=y$.
\end{proof}
%
%
%
\begin{lemma}\label{lem:uniqueness2}
Let $I \subseteq \R$ be an interval with $0 \in \accentset{\circ}{I}$ and $\gamma: I \to \H$ be continuous, with
\begin{equation}\label{eq:surjectivity-proof-bound} \gh{\gamma_0^{-1} \gamma_t }   \le \varrho \abs{t}^{\frac{1+\alpha}{2}} \quad \text{and} \quad \abs{ \bra{\gamma_0^{-1}  \gamma_t}^\v -t } \le \varrho^2 \abs{t}^{{1+\alpha}} \quad \text{for $t \in I$,}\end{equation}
for some $\varrho >0$. Then, there exists $\delta_2>0$ such that $[-\delta_2, \delta_2]\subseteq I$ and the following holds: for any $\delta \in (0, \delta_2]$, there is $\veps_2 = \veps_2(\delta)>0$ such that, if $x\in \ob{\veps_2}{\gamma_0}$, then 
\[\gv{\gamma_{t}^{-1}x} \le \gh{ \gamma_{t}^{-1} x},\quad \text{for some $t = t(x) \in [-\delta, \delta]$.}\]
\end{lemma}

\begin{remark}
Solutions to the \lsd satisfy~\eqref{eq:surjectivity-proof-bound}, up to restricting their interval of definition, by Lemma~\ref{lem:modulus-continuity} below.
\end{remark}

\begin{remark}[Comparison with the Euclidean case]
The above lemma, which seems intricated, becomes obvious if formulated for classical situation of a continuously differentiable curve in Euclidean space $\R^3$ (instead of $\H$), which means replacing the second condition in~\eqref{eq:surjectivity-proof-bound} by requiring $\dot \gamma^3_0 =1$ (the first condition in ~\eqref{eq:surjectivity-proof-bound} is then unnecessary). In fact, one can obtain $t=t(x)$ such that $x^3 = \gamma_t^3$.
\end{remark}


\begin{proof}
For simplicity, write throughout the proof $\delta$, $\veps$ to be specified below, yielding $\delta_2$, $\veps_2$ as in the thesis. Without any loss of generality, we assume $\gamma_0 = 0$: the general case follows from reducing to the curve $\gamma_0^{-1}\gamma_t$.
Using the group operation~\eqref{eq:heisenberg-operation}, we write for $t \in I$,
\[  (\gamma_t^{-1}x)^\v  = x^\v - \gamma_t^\v - \gamma^1_t x^2 + \gamma^2_tx^1 = x^\v - \gamma_t^\v - \gamma^1_t (x^2-\gamma^2_t) + \gamma^2_t(x^1-\gamma^1_t) .\]
The inequality $|ab| \le a^2/4 + b^2$ yields
\begin{equation}\label{eq:uniqueness-upper} \begin{split} (\gamma_t^{-1}x)^\v & \le  x^\v - \gamma_t^\v + \frac{ \abs{\gamma^1_t}^2}{4}+  (x^2-\gamma^2_t)^2 + \frac{\abs{\gamma^2_t}^2}{4} + (x^1-\gamma^1_t)^2 \\
& = x^\v - \gamma_t^\v + \frac{ \bra{\gh{\gamma_t}}^2 }{4} +  \bra{\gh{ \gamma_t^{-1} x}}^2  \\
& \le x^\v - t + 2\varrho^2 \abs{t}^{1+\alpha} +  \bra{\gh{ \gamma_t^{-1} x}}^2, \quad \text{by ~\eqref{eq:surjectivity-proof-bound},} \end{split} \end{equation}
and in the same way
\begin{equation}\label{eq:uniqueness-lower} (\gamma_t^{-1}x)^\v 
\ge x^\v - t - 2\varrho^2 \abs{t}^{1+\alpha} -  \bra{\gh{ \gamma_t^{-1} x}}^2 . \end{equation}
If $\gv{x} \le \gh{ x}$, the thesis follows choosing $t = t(x):=0$, since $\gamma_0 = 0$. Otherwise, i.e.\ if $\gv{x} > \gh{ x}$, we distinguish between the case $x^\v > \bra{\gh{ x}}^2$ and $x^\v < -\bra{\gh{ x}}^2$, taking into account the sign of $x^\v$. In the former case, we introduce the continuous function
\[ G(t):= (\gamma_t^{-1}x)^\v - \bra{\gh{  \gamma_t^{-1} x}}^2   \quad  \text{for $t \in I$.}\]
We have $G(0) = x^\v - \bra{\gh{ x}}^2 >0$. By~\eqref{eq:equivalence-distance} and the condition $x \in \ob{\veps}{0}$ we deduce $0 \le x^\v \le \c^2 \veps^2$, where $\c$ is the constant in~\eqref{eq:equivalence-distance}. Therefore, if $\delta$, $\veps$ are chosen so that
\begin{equation}\label{eq:condition-surjectivity}  [-\delta, \delta] \subseteq I\quad \text{and} \quad \c^2 \veps^2 \le \frac{\delta}{2},\end{equation}
then we have $2 x^\v \in I$ and we estimate from above, using~\eqref{eq:uniqueness-upper},
\[
 G(2x^\v) \le x^\v - (2x^\v) +  2\varrho^2 \abs{2x^\v}^{1+\alpha} = -x^\v + 4 \varrho^2 \delta^\alpha x^\v =( 4 \varrho^2 \delta^\alpha  -1) x^\v.
\]
If $\delta>0$ is such that additionally
\begin{equation}\label{eq:condition-surjectivity-2} 
4 \varrho^2 \delta^\alpha  -1 \le 0,
\end{equation}
then $G(2x^\v) \le 0$ and by continuity we deduce that, for some $t = t(x) \in (0, 2x^\v]$, $G(t) =0$, from which the thesis follows.

Arguing in a symmetric way in the case $x^\v < -\bra{\gh{ x}}^2$, i.e.\ by considering instead
\[ G(t):= (\gamma_t^{-1}x)^\v + \bra{\gh{  \gamma_t^{-1} x}}^2  \quad  \text{for $t \in I$,}\]
 and using~\eqref{eq:uniqueness-lower} instead of~\eqref{eq:uniqueness-upper} we deduce that, if $\delta$, $\veps$ satisfy~\eqref{eq:condition-surjectivity} and~\eqref{eq:condition-surjectivity-2}, there exists $t = t(x)  \in [2x^\v, 0)$ such that  $G(t) = 0$. In conclusion, to obtain the thesis, it is enough to fix $\delta_2$ such that~\eqref{eq:condition-surjectivity-2} holds (with $\delta_2$ instead of $\delta$) and then, for $\delta \in (0, \delta_2]$, let $\veps  = \veps_2(\delta)>0$ which satisfies~\eqref{eq:condition-surjectivity}.
%
\end{proof}

%
%
\begin{proposition}[Surjectivity]\label{prop:surj}
Let $p \in \H$ be a nondegenerate point for $F\in \conetwo$. 
If $I \subseteq \R$ is an interval with $0 \in \accentset{\circ}{I}$, and
\begin{itemize}
\item[(i)] $\gamma: I \to \O$ is continuous and~\eqref{eq:surjectivity-proof-bound} holds, with
\item[(ii)] $F(\gamma_t) = F(\gamma_0)$, for $t \in I$,
\end{itemize}
then there exists $\delta_3>0$, such that $[-\delta_3, \delta_3 ] \subseteq I$ and the following holds: for any $\delta \in (0, \delta_3]$, there is $\veps_3 = \veps_3(\delta)>0$ such that, if $\gamma_0 \in \ob{\veps_3}{p}$, we have
\begin{equation} \label{eq:surjectivity} \gamma([-\delta, \delta])\cap \ob{\veps_3}{p} = F^{-1}\bra{F(\gamma_0) }\cap \ob{\veps_3}{p}. \end{equation}
\end{proposition}
\begin{proof}
As usual, we write throughout the proof $\delta$, $\veps$, to be specified below, yielding $\delta_3$, $\veps_3$ such that  the thesis holds.  Let $\veps_1>0$ be provided in Lemma~\ref{lem:uniqueness1} and $\delta_2>0$ be provided by Lemma~\ref{lem:uniqueness2}. Then, if $\delta \in (0, \delta_2]$, let $\veps_2 = \veps_2(\delta)>0$ as provided by  Lemma~\ref{lem:uniqueness2}. Given any $\delta \in (0, \delta_2]$, from~\eqref{eq:equivalence-distance} and~\eqref{eq:surjectivity-proof-bound}, we have, for $t \in [-\delta, \delta]$, 
\begin{equation*}
\begin{split}
\dist(\gamma_t, \gamma_0) & \le \c\bra{ \gh{\gamma_0^{-1} \gamma_t} +  \gv{\gamma_0^{-1} \gamma_t}} \le \c\bra{ \varrho \abs{t}^{\frac{1+\alpha}{2}}+  \sqrt{ \abs{t} + \varrho^2 \abs{t}^{1+\alpha}} }\\
& \le \c\bra{ \varrho \delta^{\frac{1+\alpha}{2}}+  \sqrt{ \delta + \varrho^2 \delta^{1+\alpha}} }
\end{split}
\end{equation*}
Therefore, if $\delta$ satisfies
\begin{equation}\label{eq:condition-surg-level-sets-delta}
\delta \in (0,\delta_2] \quad \text{and} \quad \c\bra{ \varrho \delta^{\frac{1+\alpha}{2}}+  \sqrt{ \delta + \varrho^2 \delta^{1+\alpha}} } \le \frac{\veps_1}{2},
\end{equation}
we deduce that $\dist(\gamma_t, \gamma_0) \le {\veps_1}/{2}$. Hence, if $\veps$ satisfies
\begin{equation}\label{eq:condition-surg-level-sets-delta-2}
\veps \le {\veps_1}/{2} \quad \text{and} \quad \veps \le \veps_2(\delta),
\end{equation}
for $x \in F^{-1}\bra{F(\gamma_0) }\cap \ob{\veps}{p}$, Lemma~\ref{lem:uniqueness2} provides a $t = t(x) \in [-\delta, \delta]$ such that
\[\gv{\gamma_{t}^{-1}x} \le \gh{ \gamma_{t}^{-1} x}.\]
Moreover, if $\gamma_0 \in \ob{\veps}{p}$, then $\dist(\gamma_t, p) \le \dist(\gamma_t, \gamma_0)+ \dist(\gamma_0, p) \le \veps_1$, and Lemma~\ref{lem:uniqueness1} with $x$ in place of $y$ and $\gamma_t$ in place of $x$ therein, we deduce $\gamma_t = x$. Hence, if we choose $\delta_3$ such that~\eqref{eq:condition-surg-level-sets-delta} hold (with $\delta_3$ instead of $\delta$) and then, for $\delta \in (0, \delta_3]$, we choose $\veps_3(\delta)$ such that~\eqref{eq:condition-surg-level-sets-delta-2} holds, we have the inclusion $\supseteq$ in~\eqref{eq:surjectivity}, while the converse inclusion is assumption (ii).
\end{proof}
%

Putting together Theorem~\ref{thm:existence} and Proposition~\ref{prop:surj}, we have the following result concerning the local parametrization of level sets of maps $F \in C^{1,\alpha}_{\h}  (\O, \R^2)$ at nondegenerate points.

\begin{theorem}[Parametrization of level sets]\label{thm:parametrization}
Let $p \in \H$ be a nondegenerate point for $F\in \conetwo$. Then, there exists $\delta_4>0$ such that the following condition holds: for any $\delta \in (0,\delta_4]$, there is an $\veps_4 = \veps_4(\delta)$ such that,  for any $q  \in \ob{\veps_4}{p}$ there is an injective solution to the \lsd on $I = [-\delta, \delta]$  with $\gamma_0= q$ and
\begin{equation}\label{eq:surjectivity-param}  \gamma(I)\cap \ob{\veps_4}{p} = F^{-1}(F(q)) \cap \ob{\veps_4}{p}. \end{equation}
\end{theorem}

\begin{proof}
As usual, write $\delta$, $\veps$, throughout the proof, to be specified below, yielding $\delta_4$, $\veps_4>0$ such that the  thesis holds. Let $\delta_0$, $\veps_0$, $\varrho_0$ be as in Theorem~\ref{thm:existence}.  If  $\veps>0$ satisfies
\begin{equation}\label{eq:condition-veps-four-1}
\veps \le \veps_0
\end{equation}
then Theorem~\ref{thm:existence} provides $\gamma: [-\delta_0, \delta_0] \to \H$ that solve the LSDE,  with $\gamma_0 = q$ and~\eqref{eq:holder-horizontal} holds. 
Since $\gamma$ is concentrated on the level set $F^{-1}\bra{F(q)}$, in order to apply Proposition~\ref{prop:surj} with such $\gamma$ and $I = [-\delta_0, \delta_0]$, we have to ensure condition~\eqref{eq:surjectivity-proof-bound}, for some $\varrho>0$. The first inequality in~\eqref{eq:surjectivity-proof-bound} follows immediately from the first bound in~\eqref{eq:holder-horizontal}, with
\begin{equation}\label{eq:first-condition-varro} \varrho \ge \rho_0.\end{equation} The second inequality in~\eqref{eq:surjectivity-proof-bound} follows from the second bound in~\eqref{eq:holder-horizontal}: indeed it is sufficient to recall the definition of $\E_{st}$ in the ``vertical'' equation of~\eqref{eq:yde-group}, and the fact that $\abs{\E_{0t}} \le \norm{\E} \abs{t}^{1+\alpha}$, for $t \in [-\delta_0, \delta_0]$. In particular, we find for $\varrho$ the additional condition
\begin{equation}\label{eq:second-condition-varro} \varrho^2 \ge \kappa \rho_0^2.\end{equation}
 Hence, choosing $\varrho$ to satisfy~\eqref{eq:first-condition-varro} and~\eqref{eq:second-condition-varro}, Proposition~\ref{prop:surj} applies, providing a $\delta_3>0$ (with $\delta_3 \le \delta_0$) such that, for $\delta \in (0, \delta_3]$, if $q \in \ob{\veps_3(\delta)}{p}$, then~\eqref{eq:surjectivity} holds. 
Hence, if we let $\delta_4 := \delta_3$ and for $\delta \in (0,\delta_4]$ choose an $\veps_4$ satisfying~\eqref{eq:condition-veps-four-1} and $\veps_4 \le \veps_3(\delta)$, the thesis follows.
%
\end{proof}

We end this section with the following ``stability'' version of Theorem~\ref{thm:parametrization}. It is interesting to notice that here we do not use uniqueness of solutions the LSDE.

\begin{corollary}[Stability]\label{cor:stability}
Let $p \in \H$ be a nondegenerate point for $F\in \conetwo$ and let $\cur{F^n}_{n\ge 1} \subseteq \conetwo$ converge to $F$ in $\conetwo$. Then, there exist $\bar n\ge 1$, positive $\delta_4$, $\varrho_4$ such that the following holds. For any $\delta \in (0, \delta_4]$, there is $\veps_4 = \veps_4(\delta) >0$, such that, for any $n \ge n_0$ and $q^n \in \ob{\veps_4}{p}$, there is an injective solution $\gamma^n:[-\delta, \delta] \to \H$  to the \lsd associated to $F^n$, with $\gamma_0^n = q^n$,~\eqref{eq:holder-horizontal-stability} and
\begin{equation*} \gamma^n(I)\cap \ob{\veps_4}{p} = (F^n)^{-1}(F^n(q^n)) \cap \ob{\veps_4}{p} \end{equation*}
Moreover, the family $\cur{\gamma^n}_{n\ge1}$ is compact with respect to uniform convergence  and any limit point is an injective solution $\gamma$ to the \lsd associated to $F$, which satisfies $\gamma_0 = q \in \ob{\veps_4}{p}$,~\eqref{eq:holder-horizontal} and~\eqref{eq:surjectivity-param}.
\end{corollary}

\begin{proof}
The argument is a combination of Corollary~\ref{cor:compactness} and a simple constants-chasing throughout the results of this section. Indeed, let $\bar{n}$, $\delta_0$ $\veps_0$, $\varrho_0$ be as in Corollary~\ref{cor:compactness}, and let $\varrho_4 = \varrho_0$. Moreover, notice that if $\veps_1$ is chosen so as to have strict inequality in~\eqref{eq:condition-injectivity}, then by convergence of $F^n$ to $F$ in $\conetwo$ we have that, for $n$ large enough (without loss of generality $n \ge \bar{n}$), the thesis of Lemma~\ref{lem:injective} holds for $F^n$, with such $\veps_1$ (as well as for $F$).

With such choices of $\delta_0$, $\veps_0$, $\varrho_0$ and $\veps_1$, if we follow throughout the proof of Theorem~\ref{thm:parametrization} with $F^n$ in place of $F$, we see that  the thesis still holds, using~\eqref{eq:holder-horizontal-stability} instead of~\eqref{eq:holder-horizontal}, provided that $\delta_3$ and $\veps_3(\delta)$ can be made independent of $n$, for $n \ge \bar{n}$ (as well as for $F$).  To show this fact, we notice first that  choosing $\rho>0$ such that~\eqref{eq:first-condition-varro} and~\eqref{eq:second-condition-varro} hold, we have that~\eqref{eq:surjectivity-proof-bound} holds with $\gamma^n$ in place of $\gamma$, for $n \ge \bar n$. As a consequence, in the proofs of Lemma~\ref{lem:uniqueness2} and Proposition~\ref{prop:surj}, the conditions on $\delta_2$, $\veps_2(\delta)$, i.e.~\eqref{eq:condition-surjectivity} and~\eqref{eq:condition-surjectivity-2}, as well as those on $\delta_3$ and $\veps_3(\delta)$, i.e.~\eqref{eq:condition-surg-level-sets-delta} and~\eqref{eq:condition-surg-level-sets-delta-2} can be satisfied uniformly in $n$, for $n \ge \bar n$ (as well as for $F$).
\end{proof}

%
%

\section{Uniqueness of solutions}

In this section, we prove that solutions to the \lsd are unique, for small times, i.e.\ until the first time they leave a sufficiently small neighbourhood of $p$. First, we give a basic result on the modulus of continuity of any solution to the LSDE, showing in particular  that~\eqref{eq:surjectivity-proof-bound} above holds (for small times).

\begin{lemma}[Local modulus of continuity]\label{lem:modulus-continuity}
Let $p \in \H$ be a nondegenerate point for $F\in \conetwo$. Given a solution $\gamma: I \to \O$ to the LSDE, there is an  $\veps>0$ such that, if $\gamma_t \in \ob{\veps}{p}$, for all $t \in I$, then
\begin{equation}\label{eq:modulus-continuity-general}\limsup_{\substack{s,t \in I\\ \abs{t-s} \to 0}} \frac{ \gh{\gamma_s^{-1} \gamma_t}}{\abs{t-s}^{\frac{1+\alpha}{2}}} < \infty \quad \text{and} \quad \limsup_{\substack{s,t \in I\\ \abs{t-s} \to 0}} \frac{ \dist(\gamma_s, \gamma_t)}{\abs{t-s}^{1/2}}< \infty. \end{equation}
In particular, there exists a $\varrho>0$ such that~\eqref{eq:surjectivity-proof-bound} holds, up to replacing $I$ therein with some smaller interval $J \subseteq I$ (but still with $0 \in \accentset{\circ}{J}$). 
\end{lemma}

\begin{proof}
From  the ``horizontal'' equation in~\eqref{eq:yde-group} and inequality~\eqref{eq:error-taylor-three-points} applied to $x =\gamma_s$, $y= \gamma_t$,  we have
\begin{equation*}\begin{split}
\gh{\gamma_s^{-1} \gamma_t}   & = \abs{\nablah  F(p)^{-1} \bra{  \rt(p,\gamma_t)-\rt(p,\gamma_s)} } \\
& \le \abs{\nablah F(p)^{-1}} \abs{ \rt(p,\gamma_t)-\rt(p,\gamma_s)}\\
& \le \c \abs{\nablah F(p)^{-1}} \norm{\nablah  F}_{\alpha, \ob{2 \! \c \veps }{p}} \sqa{ r^\alpha \gh{\gamma_s^{-1}\gamma_t} + \bra{\gv{\gamma_s^{-1}\gamma_t} }^{1+\alpha}}
\end{split}\end{equation*}
If $\veps$ satisfies
\[ \c \abs{\nablah  F(p)^{-1} } \norm{\nablah F}_{\alpha,  \ob{2 \! \c \veps }{p}}   \veps^\alpha \le \frac 1 2, \]
we deduce
\begin{equation}\label{eq:horizontal-bound-vertical} \gh{\gamma_s^{-1} \gamma_t} \le 2 \c \abs{\nablah  F(p)^{-1} } \norm{\nablah F}_{\alpha,  \ob{2 \! \c \veps }{p}} \bra{\gv{\gamma_s^{-1}\gamma_t} }^{1+\alpha}.\end{equation}
Since we are interested in the limit as $\abs{t-s} \to 0$, we assume that $\abs{t-s} \le 1$. By~\eqref{eq:yde-group} and~\eqref{eq:error}, we have
 \begin{equation}\label{eq:bound-vertical}
  \abs{\bra{\gamma_s^{-1} \gamma_t}^\v}  \le \abs{t-s} + \norm{\E} \abs{t-s}^{1+\alpha} \le \bra{ 1 + \norm{\E} } \abs{t-s},\,  \end{equation}
  i.e.,   $\gv{\gamma_s^{-1} \gamma_t} \le \sqrt{ 1 + \norm{\E} } \abs{t-s}^{1/2}$,
 which together with~\eqref{eq:horizontal-bound-vertical} yields
\[ \limsup_{\substack{s,t \in I\\ \abs{t-s} \to 0}} \frac{ \gh{\gamma_s^{-1} \gamma_t}}{\abs{t-s}^{\frac{1+\alpha}{2}}}  \le   2 \c \abs{\nablah  F(p)^{-1} } \norm{\nablah F}_{\alpha,  \ob{2 \! \c \veps}{p}} \bra{1 + \norm{\E}}^{(1+\alpha)/2}, \]
which is the first bound in~\eqref{eq:modulus-continuity-general}. The second one  follows then from~\eqref{eq:bound-vertical} and~\eqref{eq:equivalence-distance}. Finally, the claim on the validity of the two inequalities in~\eqref{eq:surjectivity-proof-bound} follows respectively from the first inequality in~\eqref{eq:modulus-continuity-general}, with $s =0$, and from the definition of $\E_{0t}$ and $\norm{\E}$.
%
\end{proof}


\begin{theorem}[Local uniqueness]\label{thm:loc_uniqueness_I}
Let $p \in \H$ be a nondegenerate point for $F\in \conetwo$. Given solutions $\gamma, \bar{\gamma}: I \to \H$ to the \lsd with 
$\gamma_{t_0}=\bar{\gamma}_{t_0} \in \ob{\veps}{p}$, for some $t_0 \in I$, then there is $\veps>0$ such that
\begin{equation}\label{eq:coincidecn} \cur{ t \in I\, : \, \gamma_t = \bar{\gamma}_t}\end{equation}
contains the connected component of $t_0$ in $\cur{ t \in I\, : \, \gamma_t \in \ob{\veps}{p}}$.
\end{theorem}

\begin{proof}
There is no loss in generality if we prove the thesis for $t_0 = 0$, the general case following reducing to solutions of the \lsd $t \mapsto \gamma_{t_0 + t}$, $t \mapsto \bar \gamma_{t_0 + t}$. Lemma~\ref{lem:injective} and  Lemma~\ref{lem:modulus-continuity} give that, possibly up to replacing $I$ with a smaller neighbourhood of $t_0=0$, both $\gamma$ and $\bar\gamma$  satisfy~\eqref{eq:local-injectivity} as well as~\eqref{eq:surjectivity-proof-bound} (without loss of generality, with the same constants $\delta$ and $\varrho$).

Proposition~\ref{prop:surj} applied to $\gamma$ (respectively, to $\bar \gamma$) provide some $\delta_3$ and $\veps_3(\delta)$  (respectively, $\bar{\delta}_3$ and $\bar{\veps}_3(\delta)>0$). If $\delta>0$ satisfies
\begin{equation}\label{condition-uniqueness-delta}
\delta \le \delta_3 \quad \text{and} \quad \delta \le \bar{\delta}_3 
\end{equation}
and $\veps>0$ satisfies
\begin{equation}\label{condition-uniqueness-eps}
\veps \le \veps_3(\delta) \quad \text{and}  \quad \veps \le \bar{\veps}_3 (\delta),
\end{equation}
then Proposition~\ref{prop:surj} gives
\[ \gamma([-\delta, \delta])\cap \ob{\veps}{p} = F^{-1}\bra{F(\gamma_0) }\cap \ob{\veps}{p} = \bar\gamma([-\delta, \delta])\cap \ob{\veps}{p}.\]

In particular, for any $t \in [-\delta, \delta]$, there is $\bar{t} \in [-\delta, \delta]$ such that $\bar{\gamma}_{\bar{t}} = \gamma_{t}$. Such $\bar{t}$ is unique by injectivity of $\bar{\gamma}$, hence the function $t \mapsto \bar t$ is well defined on $[-\delta, \delta]$, with $\bar t = 0$ for $t =0$. 

By the ``vertical'' equation in~\eqref{eq:yde-group}, for $s$, $t \in [-\delta, \delta]$, we have
 \begin{equation}\label{eq:vertical-uniqueness} t-s + \E_{st} = ({\gamma}_s^{-1} {\gamma}_t)^\v = (\bar\gamma_{\bar s}^{-1}\bar \gamma_{\bar t} )^\v= \bar t - \bar s + \bar{\E}_{\bar s \bar t},
 \end{equation}
 hence
\begin{equation}
\label{eq:identity-uniqueness-gamma}  \abs{ \bar t - \bar s} \le \abs{t-s} + \abs{\E_{st}} + \abs{\bar{\E}_{\bar{s}\bar t}} 
 \end{equation}
From~\eqref{eq:error}, we estimate from above
\begin{equation}\label{eq:bar-e}\abs{ \bar{\E}_{\bar s \bar t } }\le \norm{\bar{\E}} \abs{ \bar t - \bar s }^{1+\alpha} \le \norm{\bar{\E}} (2\delta)^\alpha \abs{ \bar t - \bar s },\end{equation}
If $\delta>0$ satisfies additionally
\begin{equation}\label{condition-uniqueness-delta-1}
 \norm{\bar{\E}} (2\delta)^\alpha  \le  1/ 2,\end{equation}
from~\eqref{eq:identity-uniqueness-gamma} and~\eqref{eq:bar-e} we deduce
\begin{equation}\label{eq:uniqueness-lipschitz}  \abs{ \bar t - \bar s} \le  2 \bra{  \abs{t-s} + \abs{\E_{st}}}   \le 2\bra{1+ (2\delta)^\alpha  \norm{{\E}}} \abs{t-s},\end{equation}
using the bound $\abs{ \E_{st } } \le \norm{{\E}} (2\delta)^\alpha \abs{ t - s }$. Combining   ~\eqref{eq:bar-e} and~\eqref{eq:uniqueness-lipschitz}, we find
\[
\abs{ \bar{\E}_{\bar s \bar t}} \le \norm{\bar\E} \abs{ \bar t - \bar s}^{1+\alpha} \le\norm{\bar\E} \sqa{2\bra{1+ (2\delta)^\alpha  \norm{{\E}}}}^{1+\alpha} \abs{ t -  s}^{1+\alpha},
\]
thus
\begin{equation} \label{eq:estimate-error-e-phi}
\limsup_{\substack{s, t \in I \\ \abs{t-s} \to 0}}\frac{ \abs{ \bar{\E}_{\bar s \bar t}} }{\abs{t-s}} = 0.
\end{equation}
Dividing by $t-s$ the leftmost and the rightmost sides in~\eqref{eq:vertical-uniqueness}, we find
\[\frac{\bar t - \bar s }{t-s} = 1 + \frac{ \E_{st}}{t-s} - \frac{\bar{\E}_{\bar s \bar t}}{t-s}.\]
Letting $s \to t$, using~\eqref{eq:estimate-error-e-phi}, we deduce that $t \mapsto \bar t$ is differentiable at any 
$t \in (-\delta, \delta)$, with derivative identically $1$:  it follows that $t = \bar{t}$. 

Hence, we have that if $\delta$  satisfies~\eqref{condition-uniqueness-delta},~\eqref{condition-uniqueness-delta-1} and $\veps$ satisfies~\eqref{condition-uniqueness-eps},  with $\gamma_0 = \gamma_0 \in \ob{\veps}{p}$, then $I \supseteq [-\delta, \delta]$. In particular, in view of~\eqref{eq:local-injectivity}, we have that $\gamma_t = \bar \gamma_t$ coincide up to the first time they leave $\ob{\delta^{1/2}/\varrho}{\gamma_0}$. Hence, if we further assume $\veps \le {\delta^{1/2}}/\bra{{2\varrho}}$, then $\ob{\veps}{p} \subseteq \ob{\delta^{1/2}/\varrho}{\gamma_0}$, so that~\eqref{eq:coincidecn} must contain all the connected component of $t=0$ in  $\cur{ t \in I\, : \, \gamma_t \in \ob{\veps}{p}}$.
%
\end{proof}

\section{Area formula}

In this section, we establish an integral formula for the spherical Hausdorff measure 
of ``vertical'' curves satisfying conditions akin to those of the LSDE.  
Then, we obtain the corresponding area formula for level sets of $F \in \conetwo$, in a neighbourhood of a  nondegenerate point.

 We recall the definition of the $2$-dimensional  spherical Hausdorff measure of a set $\A \subseteq \H$. For an $\veps>0$, set
\begin{eqnarray*}
\cS_{\dist,\veps}^2 (\A) &:= &\inf\cur{ \sum_{i \ge 1} \beta_\dist \,r_i^2\, \colon  \, \A \subseteq \bigcup_{i \ge 1} \ob{r_i}{p_i}, \, \text{and $r_i \le \veps$, for every $i \ge 1$} }, \quad \text{where}\\
 \beta_\dist &:= &\sup_{\dist(0,y)\le 1}\cL^1\lls\set{\sigma \in\R: (0,0,\sigma)\in \ob{1}{y}}\rrs,
\end{eqnarray*}
the $\inf$ running among all families of balls covering $\A$ and $\cL^1$ denoting Lebesgue measure. Then, define the $2$-dimensional spherical Hausdorff measure as
\begin{equation}\label{eq:hausdorff} \cS_\dist^2(\A):= \sup_{\veps>0}\cS_{\dist,\veps}^2 (\A) = \lim_{\veps \to 0^+ } \cS_{\dist,\veps}^2 (\A).\end{equation}
\begin{theorem}[Area formula]\label{thm:area}
Let $I\subseteq \R$ be an interval, let $\gamma: I \to \H$ be injective and such that
\begin{equation}
 \label{eq:curve-area-formula}\gh{ \gamma_s^{-1}\gamma_t }   \le   \sqrt{ |t-s|  \omega(\abs{t-s}) }
 \qandq \abs{ (\gamma_s^{-1}\gamma_t)^\v - \int_s^t \vth(\tau) \d\tau } \le |t-s| \omega(
 \abs{t-s})
\end{equation}
for every $s,t\in I$, where $\vartheta: I \to \R$ is continuous, $\omega: [0, \infty) \to [0, \infty)$ is non-decreasing and $\omega(0^+)=0$. 
Then, for every Borel  $\A \subseteq \H$ and bounded Borel function $u: \O \to \R$, we have
\begin{equation*}
 \cS_\dist ^2 \bra{ \gamma(I) \cap \A} = \gamma_\sharp ( \abs{\vth} \cL^1 \res I)(\A)
\quad \text{and}\quad
\int_{\gamma(I)} u  \d \cS_\dist^2 = \int_I u\bra{\gamma_\tau} \abs{ \vth (\tau)} \d \tau.
\end{equation*}
\end{theorem}


\begin{proof}
Without loss of generality, we assume that $I$ is compact (otherwise, the argument follows by covering $I$ with compact intervals). The proof relies on the measure-theoretic area formula \cite[Theorem~11]{magnani_measure-theoretic_2015} on the metric space $(\H, \dist)$, which reads
\begin{equation}\label{eq:area-formula-abstract} \gamma_\sharp ( \abs{\vth} \cL^1 \res I)(\A)= \int_\A \theta(x) \d  \cS_\dist^2 (x), \quad \text{for $\A \subseteq \H$ Borel,}\end{equation}
$\theta(x)$ is the (upper) spherical Federer density of  $\gamma_\sharp ( \abs{\vth} \cL^1 \res I)$ at $x\in \H$.
This density, introduced in \cite{magnani_measure-theoretic_2015},
can be equivalently defined as follows 
\begin{equation}\label{eq:abstract-federer-density}
\theta(x)
= \sup \cur{ \limsup_k \frac{{ \gamma_\sharp( \abs{\vth} \cL^1 \res I)(\ob{\varrho_k}{y_k})}}{\beta_\dist \varrho^2_k} \colon {\dist(y_k, x) \le \rho_k\to 0}}.
\end{equation}
Once we prove that the assumptions in \cite[Theorem~11]{magnani_measure-theoretic_2015} are satisfied, it will suffice to show that $\theta(x) = 1$, for $\cS_\dist^2$-a.e.\ $x \in \gamma(I)$.

\emph{Claim: $\gamma(I)$ has finite $\cS_\dist^2$ measure.}  We prove  the inequality
\beq\label{eq:S^2_dgammaEc1} 
 \cS_\dist^2 \bra{ \gamma(I) } \le \c \int_I  \abs{ \vth(\tau)} \d\tau,
\eeq
for some constant $\c>0$. 
To show~\eqref{eq:S^2_dgammaEc1}, from~\eqref{eq:curve-area-formula},~\eqref{eq:equivalence-distance} and the inequality $(a+b)^2 \le 2 \bra{a^2 + b^2}$, we obtain  that, for some constant $\c >0$, 
\begin{equation*}
 {\dist(\gamma_t, \gamma_s)}^2  \le  \c  \bra{ \gh{ \gamma_s^{-1}\gamma_t} + \gv{ {\gamma_s^{-1}\gamma_t}}}^2 \le 2 \c  \bra{ 2 \abs{t-s} \omega(\abs{t-s}) + \int_s^t \abs{ \vth(\tau)} \d\tau } 
\end{equation*}
For $\delta>0$, choose any partition $t_0< \ldots < t_n$ of $I$ with $\sup_{i=1, \ldots, n} \abs{t_{i} - t_{i-1}} \le \delta$. Since $\omega$ is non-decreasing, we have from the above inequality
\begin{equation}\label{eq:bound-upper-area-formula-partition}\begin{split}
\diam(\gamma([t_{i-1},t_i]))^2&\le 2\c \bra{ 2
 |t_i-t_{i-1}| \omega(\delta)  + \int_{t_{i-1}}^{t_i} |\vth(\tau)| \d\tau} \\
 & \le 2\c \delta \bra{ 2\omega(\delta)   + \sup_{\tau \in I}{\abs{\vth(\tau)}}}.
 \end{split}
\end{equation}
Hence, if we let $\veps(\delta)$ denote the square root of the term in the last line in~\eqref{eq:bound-upper-area-formula-partition}, considering the covering $\gamma(I) \subseteq \bigcup_{i=1}^n \ob{\diam(\gamma([t_{i-1},t_i]) }{\gamma_{t_{i}}}$, we obtain the bound from above
\begin{equation*}\begin{split}
\cS^2_{\dist,\veps(\delta)}(\gamma(I))& \le 2 \beta_\dist \c \sum_{i=1}^n  \bra{2|t_i-t_{i-1}| \omega(\delta) + \int_{t_{i-1}}^{t_i} |\vth(\tau)| \d\tau} \\
& \le 2 \beta_\dist \c \bra{ 2 \cL^1(I) \omega(\delta)+\int_I|\vth(\tau)|\d\tau},
\end{split}\end{equation*}
As $\delta\to0^+$, since $\veps(\delta) \to 0^+$, we obtain~\eqref{eq:S^2_dgammaEc1}  (with $2\beta_\dist \c$ instead of $\c$).

{\emph{Claim: $\gamma_\sharp ( \abs{\vth} \cL^1\res I) \ll \cS^2\res \gamma(I)$.}} To this aim, we decompose 
\begin{equation}\label{eq:decomposition-gamma} \gamma_\sharp ( \abs{\vth} \cL^1):= \sum_{n \in \Z} \gamma_\sharp ( \abs{\vth} \cL^1 \res I_n^+ ) +  \gamma_\sharp ( \abs{\vth} \cL^1\res I_n^-),\end{equation}
where $I_n^+ := \cur{t \in I\, : \, \vth \in [2^{n}, 2^{n+1})}$ and similarly $I_n^- := \cur{t \in I\, : \, \vth \in (-2^{n+1}, -2^{n}]}$. Notice that the set $\cur{\abs{\vth} =0}$ is negligible with respect to $\abs{\vth} \cL^1\res I$.

We fix $n \in \Z$ and prove that $\gamma_\sharp ( \abs{\vth} \cL^1 \res I_n^+ )\ll \cS^2\res \gamma(I)$ (the case $I_n^-$ is analogous). This follows from the following quantitative injectivity inequality
\begin{equation}\label{eq:injectivity-area-formula} \abs{t-s} \le \c\,  \dist(\gamma_s, \gamma_t)^2, \quad \text{for $s \in I_n^+$, $t \in I$ with $\abs{t-s} \le \delta$.}\end{equation}
for some positive $\c = \c(n)$, $\delta = \delta(n)$. Taking it for the moment for granted, we deduce the claim. From~\eqref{eq:injectivity-area-formula}, one has that  $\gamma$ is locally injective on $I_n^+$, i.e.\ $\dist(\gamma_s, \gamma_t) >0$ if $0 < \abs{t-s} < \delta$ and
\begin{equation}\label{eq:transformation-diameters}  \begin{split} { \diam\bra{ \gamma^{-1}(B) \cap I_n^+} }& \le \c\,   \diam\bra{ \gamma\bra{ \gamma^{-1}(B) \cap I_n^+} }^2\\
& \le\c  \, \diam\bra{ B  \cap \gamma(I_n^+) }^2 \le \c \, \diam\bra{B \cap \gamma(I)}^2,
\end{split}\end{equation}
once $\diam(B) \le \delta$. Let now $\A \subseteq \H$ be such that $\cS^2_\dist\bra{\A \cap  \gamma(I)}=0$. By definition~\eqref{eq:hausdorff}, for any $\varrho >0$ we can find $(\ob{r_i}{p_i})_{i \ge 1}$,  such that $\A\cap\gamma(I) \subseteq \bigcup_{i \ge 1} \ob{r_i}{p_i}$, $r_i \le \delta/2$ for every $i \ge 1$ and 
\[ \sum_{i \ge 1} \beta_\dist r_i^2 \le \varrho.\]
The family $\gamma^{-1} \bra{ \ob{r_i}{p_i}} \cap I_n^+$, for $i \ge 1$ provides a covering of $\gamma^{-1}(\A) \cap I_n^+$, and~\eqref{eq:transformation-diameters} implies that each diameter is smaller than $\c \delta^2$, and
\[ \sum_{i\ge 1} \diam\bra{\gamma^{-1} \bra{ \ob{r_i}{p_i}} \cap I_n^+ } \le \c \sum_{i \ge 1}  (2r_i)^2 \le \frac{4 \c}{\beta_\dist} \sum_{i \ge 1} \beta_\dist r_i^2 \le  \frac{4 \c}{\beta_\dist}\varrho. \]
Since $\varrho>0$ is arbitrary, we deduce that the one-dimensional Lebesgue measure of $\gamma^{-1}(\A) \cap I_n^+$ is negligible, and so  $\gamma_{\sharp} ( \abs{\vartheta} \cL^1\res I_n^+ ) (\A) \le 2^n \gamma_{\sharp} ( \cL^1\res I_n^+ )(\A) = 0$.

{\emph{Proof of~\eqref{eq:injectivity-area-formula}:}} let $\sigma: [0, \infty) \to [0, \infty)$ be a non-decreasing modulus of continuity for $\vartheta$ (recall that $\vartheta: I \to \R$ is continuous and $I$ is compact).  If $\delta>0$ satisfies $\sigma(\delta) \le 2^{n}/2$, then for $s \in I_{n}^+$, $t \in I$ with $\abs{t-s} \le \delta$, from
\[ \int_s^t \abs{\vartheta(\tau)- \vartheta(s)} \d \tau \le \sigma(\delta)\abs{t-s}\]
we deduce, 
\[\int_s^t \vartheta(\tau) \d \tau \ge \vartheta(s) \abs{t-s} - \sigma(\delta) \abs{t-s} \ge 2^{n}\abs{t-s}/2,\]since $\vartheta(s) \ge 2^{n}$. By the second inequality in~\eqref{eq:curve-area-formula} and~\eqref{eq:equivalence-distance},
\[ 2^{n}\abs{t-s}/2 \le \int_s^t \vartheta(\tau) \d \tau - \bra{\gamma_s^{-1} \gamma_t}^\v + \bra{\gamma_s^{-1} \gamma_t}^\v \le \abs{t-s}\omega(\abs{t-s}) + \c \dist(\gamma_s, \gamma_t)^2,\]
hence if $\delta>0$ satisfies also $\omega(\delta) \le 2^{n}/4$, we obtain~\eqref{eq:injectivity-area-formula} with $2^{2-n} \c$ instead of $\c$.

{\emph{Measure-theoretic area formula.}}
Up to considering the Borel regular extension of the measure $\gamma_\sharp ( \abs{\vth} \cL^1 \res I)$, i.e.\ letting
\[
E \mapsto \inf_{\substack{A\supset E \\ A\, \text{Borel}}} \int_{\gamma^{-1}(A)}|\vth(\tau)|\d\tau, \quad \text{for $E \subseteq \H$,}
\]
we are now in a position to apply the measure-theoretic area formula \cite[Theorem~11]{magnani_measure-theoretic_2015}, so that~\eqref{eq:area-formula-abstract} holds. To conclude, we show that the spherical Federer density  $\theta(x)$ in~\eqref{eq:abstract-federer-density}  is $1$, for $\cS^2_\dist$-a.e\ $x \in \gamma(I)$. In view of~\eqref{eq:decomposition-gamma}, it is enough to assume $x = \gamma_s$, with $s \in I_n^+$, for some $n \in \Z$ (the case $s \in I_n^-$ being analogous). To simplify notation, we let in what follows $s = 0$ and assume that $x = \gamma_0 = 0$ (the general case can be reduced to this situation by considering the curve $t \mapsto \gamma_s^{-1} \gamma_{s + t}$). Then, we have $\vartheta(0) \in [2^{n}, 2^{n+1})$, in particular $\vartheta(0)>0$. For $\veps>0$, given $\cur{y_k}_{k\ge1} \in \H$, $\cur{r_k}_{k\ge 1}$ with $r_k >0$, $d(0,y_k) \le \veps r_k$ and $r_k \to 0$, as $k \to \infty$, we compute
\begin{equation*}
 \limsup_{k \to \infty} \frac{\gamma_\sharp( \abs{\vth} \cL^1 \res I)(\ob{\veps r_k}{y_k})}{\bra{\veps r_k}^2} = \limsup_{k \to \infty} \frac  1 {\bra{\veps r_k}^2} \int_I \abs{\vth(\gamma_\tau)} \chi_{\ob{\veps r_k}{y_k}}(\gamma_\tau) \d \tau.\end{equation*}
By compactness, we assume
\begin{equation*}
 \lim_{k \to \infty} \delta_{r_k^{-1}} y_k = y \in\ob{\veps}{0}.\end{equation*}
Recall that from~\eqref{eq:injectivity-area-formula} with $s=0$, $\gamma_0 = 0\in I_n^+$,
 \begin{equation}\label{eq:bound-tau} \abs{\tau} \le \c\,  \dist(0, \gamma_\tau)^2, \quad \text{for $\tau \in I$ with $\abs{\tau} \le \delta$,}\end{equation}
 for some positive $\delta$, $\c$  (without loss of generality, $\c$ can be chosen arbitrarily large). Since $\gamma$ is injective and continuous, there exists an $\eta>0$ such that, for every $\tau \in I$, $\abs{\tau} \ge \delta$, then $\dist(0, \gamma_\tau) > \eta$ (otherwise, by compactness, one would obtain a point $\tau \in I$  with $\abs{\tau} \ge \delta$ and $\gamma_\tau = 0$). In view of the inclusion $\ob{\veps r_k}{y_k} \subseteq \ob{2\veps r_k}{0}$, if $k$ is sufficiently large, then $2\veps r_k  \le \eta$ and we obtain, for such $k$, the identity
\begin{equation}\label{eq:integral-change-variables}\begin{split} \int_I &\abs{\vth(\gamma_\tau)} \chi_{\ob{\veps r_k}{y_k}}(\gamma_\tau) \d \tau  = \int_{-\delta}^\delta \abs{\vth(\gamma_\tau)} \chi_{\ob{\veps r_k}{y_k}}(\gamma_\tau) \d \tau \\
& = \int_{\cur{ \abs{\tau} \le 4\c \veps^2 r_k^2}} \abs{\vth(\gamma_\tau)} \chi_{\ob{\veps r_k}{y_k}}(\gamma_\tau) \d \tau  \quad \text{by~\eqref{eq:bound-tau} with $\dist(0, \gamma_\tau) \le 2\veps r_k$, $\tau \le \delta$.}\\
& = r_k^2  \int_{\cur{ \abs{\sigma} \le 4\c \veps^2 }} \abs{\vth(\gamma_{r_k^2 \sigma})} \chi_{\ob{\veps r_k}{y_k}}(\gamma_{r_k^2\sigma}) \d \sigma \quad \text{by substitution $\sigma := \tau/r_k^2$.}
\end{split}
\end{equation}
We have $\gamma_{r_k^2\sigma} \in \ob{\veps r_k}{y_k}$ if and only if $\delta_{r_k^{-1}} \gamma_{r_k^2 \sigma} \in \ob{\veps}{\delta_{r_k^{-1}} y_k}$, and
\[ \lim_{k \to \infty} \delta_{r_k^{-1}} \gamma_{r_k^2 \sigma} = \lim_{k \to \infty} ( \gamma_{r_k^2 \sigma}^\h/r_k, \gamma_{r_k^2 \sigma}^\v / r_k^2) = (0,0,\vartheta(0) \sigma) \in \H\]
by~\eqref{eq:curve-area-formula} with $s = 0$, $\gamma_0 =0$, $t = r_k^2 \sigma$. Therefore,   as $k \to \infty$, the functions $\sigma \mapsto \chi_{\ob{\veps r_k}{y_k}}(\gamma_{r_k^2\sigma})$ converge pointwise to the characteristic function of the set, 
\[  \cur{ \sigma \in\R: (0,0, \vartheta(0) \sigma)\in \ob{\veps}{y}},\]
with the possible exception of the points $\sigma$ such that $(0,0,\vartheta(0) \sigma) \in \partial \ob{\veps}{y}$. By Fatou lemma, from~\eqref{eq:integral-change-variables}, we have
\begin{equation*}\begin{split} \limsup_{k \to \infty} \frac{\gamma_\sharp( \abs{\vth} \cL^1 \res I)(\ob{\veps r_k}{y_k})}{r_k^2} & \le \int_{\cur{ \abs{\sigma} \le 4\c\veps }} \abs{\vth(0)} \chi_{\ob{\veps}{y}}((0,0, \vartheta(0) \sigma)) \d \sigma \\
 & = \int_\R \abs{\vth(0)} \chi_{\ob{\veps}{y}}((0,0, \vartheta(0) \sigma)) \d \sigma \\
& = \cL^1\lls\set{\sigma \in\R: (0,0,\sigma)\in \ob{\veps}{y}}\rrs \le \beta_\dist.
\end{split}\end{equation*}
Dividing by $\veps^2 \beta_\dist$ and letting $\veps=1$, we deduce $\theta(0) \le 1$. To show the converse inequality, let $\veps>1$, choose a maximizing sequence $\cur{y_n}_{n\ge1}$ for $\beta_\dist$ and any infinitesimal sequence $\cur{r_k}_{k\ge1}$. For fixed $n \ge 1$, let $y_k := \delta_{r_k} y_n$, so that $\dist(0, y_k) \le r_k \le \veps r_k$. From Fatou lemma,
\begin{equation}\label{eq:final-area-inequality2}\begin{split} \liminf_{k \to \infty} \frac{\gamma_\sharp( \abs{\vth} \cL^1 \res I)(\ob{\veps r_k}{y_k})}{r_k^2} & \ge \int_{\cur{ \abs{\sigma} \le 4\c\veps }} \abs{\vth(0)} \chi_{\openball{\veps}{y}}((0,0, \vartheta(0) \sigma)) \d \sigma \\
& \ge \cL^1\lls\set{\sigma \in\R: (0,0,\sigma)\in \ob{1}{y}}\rrs \ge \beta_\dist - o(1)
\end{split}\end{equation}
being $o(1)$ infinitesimal as $n \to \infty$. Taking $\veps_n \to 1^+$ and dividing~\eqref{eq:final-area-inequality2} by $\veps^2 \beta_\dist$ with $\veps = \veps_n$, we see that $\theta(0) \ge 1 - o(1)$ as $n \to \infty$, hence the thesis.
\end{proof}


\begin{corollary}[Area formula for LSDE]\label{cor:area-lsde}
Let $p \in \H$ be a nondegenerate point for $F\in \conetwo$, and let $\gamma: I \to \H$ be an injective solution to the \lsd with $\gamma_0 = q$. If $\dist(p,\gamma(I))$ is sufficiently small, then $\gamma$ is parameterized by arc-length, with respect to $\cS^2_\dist$, i.e.\ 
\[ \cS^2_\dist( \gamma(J) ) = \cL^1(J), \quad \text{for every closed $J \subseteq I$}.\]
Moreover, if for some $\veps>0$, one has
\[\gamma(I)\cap \ob{\veps}{p} = F^{-1}(F(q)) \cap \ob{\veps}{p},\]
then for every $\A \subseteq\ob{\veps}{p}$ Borel and for every bounded Borel function $u: \ob{\veps}{p} \to \R$,
\[ \cS^2_\dist\bra{F^{-1}(F(q)) \cap \A} = \int_I \chi_{\A}(\gamma_\tau) \d \tau, \quad\text{and}\quad  \int_{F^{-1}(F(q))} u \d \cS^2_\dist  = \int_I u(\gamma_\tau) \d \tau.\]
\end{corollary}

\begin{proof}
Since the second inequality in~\eqref{eq:curve-area-formula} always holds for $\gamma$ solution to the \lsd with $\vartheta =1$ and $\omega(s)=\norm{\E} s^{\alpha}$, we  have to ensure that the first inequality in~\eqref{eq:curve-area-formula} holds (possibly with a different $\omega$). This follows e.g.\ from Lemma~\ref{lem:modulus-continuity}, if $\dist(p,\gamma(I))$ is sufficiently small.
\end{proof}

\begin{remark}\label{rem:large-interval-area}
If $\gamma$ satisfies~\eqref{eq:holder-horizontal} (or a sequence $\gamma^n$ satisfies~\eqref{eq:holder-horizontal-stability}) there is no need to ensure that $\dist(p,\gamma(I))$ is sufficiently small, for~\eqref{eq:curve-area-formula} immediately follows from~\eqref{eq:holder-horizontal}.
\end{remark}

\section{Coarea formula}

In this section, we prove a coarea formula for maps $F \in C^{1,\alpha}_{\h}  (\O, \R^2)$.

\begin{definition}
If $F:\O \to \R^2$ is $\h$-differentiable at $x\in \O$, we define the {\em horizontal Jacobian} $J_{\h}  F(x)$ of $F$ at $x$ setting $J_{\h}  F(x) :=|\det \nabla _{\h}   F(x)|$.
\end{definition}
%
%

%


%
\begin{theorem}[Coarea formula]\label{thm:coarea}
Let $F \in \conetwo$. Then, for every  $\A \subseteq \O$  Borel, 
\begin{equation}\label{eq:coarea} \int_\A J_{\h}   F \d \cL^3 = \int_{\R^2} \cS_\dist ^2\bra{\A \cap F^{-1}(z)} \d \cL^2(z),
\end{equation}
and for every bounded Borel function $u: \O \to \R$,
\begin{equation}\label{eq:coarea-functions}
  \int_\H u\,  J_{\h}   F \d \cL^3 = \int_{\R^2} \int_{F^{-1}(z)} u \d \cS_\dist^2 \d\cL^2(z). 
\end{equation}
\end{theorem}

%
%

Our proof follows the approach introduced in \cite{magnani_area_2011}, in particular a 
 blow-up argument akin to \cite[Theorem~4.1]{magnani_area_2011}, but here we  rely on the parametrization provided by the LSDE. Given $F\in C^{1,\alpha}_{\h}  (\O, \R^2)$, for $p \in \O$ and $r>0$, define  the map
\begin{equation*}
F_{p,r} : \H \to \R^2,  \quad q \mapsto F_{p,r}(q) := \frac{ F\bra{p \delta_r(q) } - F(p)}{r} \quad \text{and set} \quad F_{r}:=F_{0,r}.
\end{equation*}

\begin{lemma}\label{lem:convergence-blow}
As $r \to 0^+$, the maps $\cur{F_{r,p}}_{r>0}$ converge in $C^{1,\alpha}_{\h}  (\O, \R^2)$ to the group homomorphism $\d_{\h}  F(p)$. 
Moreover,~\eqref{eq:coarea} and~\eqref{eq:coarea-functions} hold with the map $\d_{\h}F(p)$ in place of $F$. 
\end{lemma}

\begin{proof}
Without loss of generality, we prove the thesis for $p=0$ (the general case follows from considering $q \mapsto F(pq)$). Let  $\varrho>0$ and $q \in \ob{\varrho}{0}$. Then, from~\eqref{eq:taylor-two-points} with $x = 0$, $y = \delta_r(q)$, we deduce
\[ \abs{F_{r}(q)-\d_{\h}  F(0)(q)} \le \c \norm{\nablah F}_{\alpha, \ob{\c r \varrho}{0}} \varrho r^{\alpha} \to 0, \quad \text{as $r \to 0^+$.}\]
To show convergence of the horizontal gradients, we notice first that the definition of $\nablah$ in terms of left-invariant fields $X_1$, $X_2$ yields the identity
\[ \bra{\nablah F_{r}}  q^\h = \nablah F \bra{\delta_r(q) }^\h,\]
and that $\nablah (\d_{\h}  F(0)) = \nablah F(0)$ is  constant. Then,
\[ \norm{\nablah F_{r} - \nablah F(0) }_{\alpha, \ob{\varrho}{0}}  = \norm{ \nablah F\bra{ \delta_r(\cdot) }}_{\alpha, \ob{\varrho}{0}}\le  \norm{ \nablah F }_{\alpha, \ob{r \varrho}{0}} r^{\alpha} \to 0, \quad \text{as $r \to 0^+$.}\]
Finally, to show~\eqref{eq:coarea}  and~\eqref{eq:coarea-functions}, we consider first the case $\d_{\h}F(0)(q) = q^\h$, i.e.\ $\nablah F(0) = \operatorname{Id}$ the identity matrix. For $z \in \R^2$, the level set of $z$ is precisely $\cur{(z, t) \, : \, t\in \R}$, hence coarea formula follows from Fubini's theorem and Theorem~\ref{thm:area} with $\gamma_t = (z,t)$, to obtain
\[ \cS^2_\dist(\cur{(z, t) \, \colon \, t\in \R} \cap \A) = \cL^1( \cur{t \in \R \, \colon \,  (z,t) \in \A}).\]
 If the matrix $\nablah F(0) = M$ is not the identity (but invertible) we consider the map $M^{-1}\d_{\h}  F(0)$  and we reduce to the previous case, using $J_\h (M^{-1} F)= \abs{\det M} J_\h (F)$ and $(M^{-1}F)^{-1}(z) = F^{-1}(Mz)$. When $\nablah F(0)$ is not invertible both integrands are zero (the one in the right hand side for a.e.\ $z \in \R^2$).
\end{proof}
In the next lemma, we use the homotopic invariance of the degree of a continuous map. 

\begin{lemma}[Convergence of images]\label{lem:surjective-blow}
If $p \in \H$ is a nondegenerate point for the map $F\in C^{1,\alpha}_{\h}  (\O, \R^2)$, then for every $\veps>0$ and every compact set  $K \subseteq \d_{\h}  F(p)( \openball{\veps}{p} )$, there is a $\bar{r}>0$ such that, for every $r\in [0,\bar r]$, $K \subseteq F_{p,r}( \openball{ \veps}{p})$. 
%
\end{lemma}

\begin{proof}
Without loss of generality, we let $p=0$. We let $D_\veps := \openball{\veps}{0} \cap \cur{ x^\v = 0}$ which is a bounded open set in $\R^2$, identified with $\cur{x^\v =0}$, and set $\bar D_\veps := \ob{\veps}{0} \cap \cur{ x^\v = 0}$ and $\partial D_\veps := \bar D_\veps \setminus D_\veps$. 
Denote by $H_0: \bar D_\veps \to \H$ (resp.\ $H_r$, for $r>0$) the restriction of $\d_\h F(0)$ (resp.\ $F_r$) to $\bar D_\veps$. The map $H: [0,\infty) \times \bar D_\veps \to \H$, $(r,x) \mapsto H_r(x)$, is continuous in both variables, by uniform convergence of $F_r$ towards $\d_\h F(0)$ (Lemma~\ref{lem:convergence-blow}). 

For any compact $C \subseteq D_\veps$, by injectivity of $H_0$ and continuity of $H$, there is an $\bar{r} = \bar{r}(C)>0$ such that, for $r \in [0,\bar{r}]$, $H_r(C) \cap H_r(\partial D_\veps) = \emptyset$: otherwise, one could find sequences $r_k \to 0^+$, $x_k \in C$,  $y_k \in \partial D_\veps$ with $H_{r_k}(x_k)= H_{r_k}(y_k)$ and, by compactness and continuity, limit points $x \in C$, $y \in \partial D_\veps$ with  $H_{0}(x)= H_{0}(y)$. 

Given a compact $K \subseteq \d_{\h}  F(0)( \openball{\veps}{0} )$, we let $C =  H_0^{-1} (K)$. From $H_0(C) = \d_{\h}  F(0)( C ) = K$, we deduce that $z \in K$ implies $z \notin H_r(\partial D_\veps)$, for $r \le \bar r$. By homotopy invariance of the degree of a continuous map, it follows that  $\deg(H_r, \bar{D}_\veps, z) = \deg(H_0, \bar{D}_\veps, z) = \operatorname{sign} \det \nablah F(0) \neq 0$, hence there is an $x \in D_\veps$ with $z=H_r(x) = F_r(x)$. 
\end{proof}

\begin{proof}[Proof of Theorem \ref{thm:coarea}]


We split the proof in several steps.

{\emph{Reduction to a density computation.}} In view of the measure-theoretic coarea formula \cite[Theorem~2.2]{magnani_area_2011}, it is sufficient to show that the density with respect to $\cL^3$ of the measure
\begin{equation*} \nu_F(\A):= \int_{\R^2} \cS_\dist ^2\bra{\A \cap F^{-1}(z)} \d \cL^2(z)\end{equation*}
coincides with $J_\h F (p)$, at $\cL^3$-a.e.\ any point $p \in \H$. The coarea inequality \cite[Theorem~4.2]{magnani_area_2011} (i.e.\ inequality $\le$ in~\eqref{eq:coarea}, up to some multiplicative factor) implies that the set
\[
\cur{ p \in \O \, \colon  J_{\h}   F (p) = 0}
\]
is $\nu_F$-negligible, hence, without loss of generality, we let $p \in \H$ be nondegenerate for $F$. To simplify notation, we assume $p =0$ (the general case follows by considering $q \mapsto F(pq)$).

{\emph{Functional density.}} Instead of proving the usual differentiation
\begin{equation}
\label{eq:usual-density}
\lim_{r \to 0^+} \frac{\nu_F (\ob{r}{0})}{r^4} = \cL^3(\ob{1}{0}) J_\h F(0),\end{equation} 
it is technically easier, but equivalent, to introduce $\veps>0$, to be specified below, and prove
\begin{equation}  \label{eq:functional-density}\lim_{r \to 0^+} \frac{1}{r^4} \int_\H u \circ \delta_r \d \nu_F =  J_\h F(0),\end{equation}
 for all continuous functions $u: \H \to [0,\infty)$ with  $\int_{\ob{\veps}{0}} u  \d\cL^3=1$ and $\supp (u) \subseteq \ob{\veps}{0}$.  Indeed, if~\eqref{eq:functional-density} holds, we let $u(x):= \c (\veps - \dist(0,x))^+ = \int_0^\veps \chi_{\ob{\varrho}{0}} d \varrho$, where $\c =\c(\veps)$ is such that $\int_{\ob{\veps}{0}} u  \d\cL^3=1$. Notice that $u\circ \delta_r = \c \int_0^\veps \chi_{\ob{r \varrho}{0}} \d \varrho$. If $0$ is a differentiation point for $\nu_F$, i.e.\ the limit in the left hand side of~\eqref{eq:usual-density} exists, which we denote by $\ell$, then~\eqref{eq:functional-density} yields
\begin{equation*}
\begin{split}
J_hF(0) & = \lim_{r \to 0^+} \frac{1}{r^4} \int_\H u \circ \delta_r \d \nu_F = \lim_{r \to 0^+} \frac{ \c }{r^4}\int_0^\veps \nu_F (\ob{r \varrho}{0}) \d \varrho \\
& =   \c \int_0^\veps \lim_{r \to 0^+} \frac{\nu_F (\ob{r \varrho}{0})}{r^4} \d \varrho = \c \int_0^\veps\ell \varrho^4 \d\varrho 
= \frac{\ell}{\cL^3(\ob{1}{0})} \int_\H u \d\cL^3 = \frac{\ell}{\cL^3(\ob{1}{0})}.
\end{split}
\end{equation*} 

\emph{Change of variables.} To show~\eqref{eq:functional-density},  
we write, for $r>0$,
\begin{equation*}\begin{split}
\frac{1}{r^4} \int_\H u \circ \delta_r \d \nu_F &=  \frac{1}{r^4}\int_{\R^2}\int_{F^{-1}(z)}{ u \circ \delta_r} \d \cS_\dist^2 \d \cL^2(z)\\
& =  \frac{1}{r^2}\int_{\R^2} \int_{F^{-1}_{r} (z/r)} u(x) \d \cS_\dist^2(x) \d \cL^2(z) \quad\text{by substitution $x \mapsto \delta_{r} x$ in $\cS^2_\dist$}\\
 & =  \int_{\R^2}  \int_{F^{-1}_{r} (z)} u \d \cS_\dist^2 \d \cL^2 (z) \quad \text{by substitution $z \mapsto zr$.}\\
\end{split}\end{equation*}
Since $\supp (u) \subseteq \ob{\veps}{0}$, we can restrict the integration over $\R^2$ to any Borel $\A \supseteq F_r(\ob{\veps}{0})$,
\begin{equation} \label{eq:coarea-change-variables}
\frac{1}{r^4} \int_\H u \circ \delta_r \d \nu_F  =   \int_{\A}   \int_{F^{-1}_{r} (z)} u \d \cS_\dist^2 \d \cL^2 (z),
\end{equation}
and similarly, by~\eqref{eq:coarea-functions} with $\d_\h F(0)$ in place of $F$ (Lemma~\ref{lem:convergence-blow}), if $\A \supseteq \d_\h F(0)(\ob{\veps}{0})$, then
\begin{equation} \label{eq:coarea-change-variables2}
\begin{aligned}
 J_\h F(0) & =  J_\h (\d _\h F(0)) =
\int_\H u \,  J_\h \bra{\d _\h F(0)} \d \cL^3
 =  \int_{\A}  \int_{\bra{\d_\h F(0)}^{-1} (z)} u \d \cS_\dist^2 \d \cL^2 (z).
\end{aligned}
\end{equation}
By uniform convergence of $F_r$ to $\d_\h F(0)$, there is an $\bar{r} = \bar{r}(\veps)>0$ such that $F_r(\ob{\veps}{0}) \subseteq \d_\h F (0) (\ob{2\veps}{0})$, hence both~\eqref{eq:coarea-change-variables} and~\eqref{eq:coarea-change-variables2} hold with $\d_\h F (0) (\ob{2\veps}{0})$ in place of $U$. 
On the other hand, Lemma~\ref{lem:surjective-blow} applied 
with $p=0$, $3\veps$ in place of $\veps$ and 
$K := \d_\h F (0) (\ob{2\veps}{0})$ gives some $\bar{r} = \bar{r}(\veps)>0$ such that
\begin{equation*}\d_\h F (0) (\ob{2\veps}{0}) \subseteq  F_r\bra{\ob{3 \veps}{0}}, \quad\text{for $r \in (0,\bar r]$,}
\end{equation*}
hence~\eqref{eq:coarea-change-variables} and~\eqref{eq:coarea-change-variables2} hold, for $r \le \bar r$, with $\A := \bigcap_{r\le \bar r} F_r\bra{\ob{3\veps}{0}}$, a choice that we fix in what follows.

\emph{Convergence of level sets.} We are in a position to apply Corollary~\ref{cor:stability} to (some subsequence of) $\cur{F_r}_{r>0}$, which converge to $\d_\h F(0)$, in $\conetwo$, as $r \to 0^+$. To simplify, we retain the notation $\cur{F_r}_{r>0}$ instead of writing e.g.\ $\cur{F_{r_n}}_{n \ge 1}$. 
We obtain some positive $\bar{r}$, $\delta_4$, $\veps_4$ and $\varrho_4$ such that, letting $I =[-\delta_4, \delta_4]$, for any $r \in (0, \bar{r}]$ and 
$q^r \in \ob{\veps_4}{0}$, there is an injective solution to the \lsd $\gamma^r: I \to \H$ associated to $F_r$, with $\gamma^r_0 = q^r$,~\eqref{eq:holder-horizontal-stability} and
\[ \gamma^r(I)\cap \ob{\veps_4}{0} = (F_r)^{-1}(F_r(q^r)) \cap \ob{\veps_4}{0}. \]
Therefore, we choose $\veps$ such that $3 \veps \le \veps_4$. Then, for any $z \in \A$, $r \in (0, \bar{r}]$, there is a $q^r\in \ob{3\veps}{0}$ such that $F_r(q^r)=z$, hence we obtain from the area formula (Corollary~\ref{cor:area-lsde} and Remark~\ref{rem:large-interval-area})
\begin{equation}\label{eq:area-in-coarea} \int_{F_r^{-1}(z)} u \d \cS^2_\dist = \int_I u(\gamma_\tau^r) \d \tau,\quad \text{for $r \in (0, \bar r]$.}\end{equation}
As $r \to 0^+$, along any (uniformly) converging subsequence $\gamma^{r_k}$, we obtain in the limit some injective solution $\gamma$ to the \lsd associated to $\d_\h F(0)$, such that~\eqref{eq:holder-horizontal} holds and
\[ \gamma(I)\cap \ob{\veps_4}{0} = \bra{\d_\h F(0)}^{-1}(z) \cap \ob{\veps_4}{0} \]
using also $z = F_{r_k}(\gamma_0^{r_k}) \to \d_\h F(0)(\gamma_0)$, by uniform convergence of $\cur{F_{r_k}}_k$. Since $u$ is continuous, we have that
\[ \lim_{r_k \to 0^+} \int_I u(\gamma_\tau^{r_k}) \d \tau= \int_I u(\gamma_\tau) \d \tau= \int_{\bra{\d_\h F(0)}^{-1}(z)} u \d \cS^2_\dist,\]
the latter equality being an application of the area formula for $\gamma$. In particular, the limit value does not depend on the subsequence $\cur{r_k}_k$, and recalling ~\eqref{eq:area-in-coarea}, we conclude that 
\[  \lim_{r \to 0^+} \int_{F_r^{-1}(z)} u \d \cS^2_\dist  = \int_{\bra{\d_\h F(0)}^{-1}(z)} u \d \cS^2_\dist,\]
for any $z \in \A$. From~\eqref{eq:coarea-change-variables} and~\eqref{eq:coarea-change-variables2}, by dominated convergence, we have
\begin{equation*}\begin{split}
\lim_{r \to 0^+} \frac{1}{r^4} \int_\H u \circ \delta_r \d \nu_F & =  \lim_{r \to 0^+}  \int_{\A}   \int_{F^{-1}_{r} (z)} u \d \cS_\dist^2 \d \cL^2 (z) \\
& =  \int_{\A}  \int_{\bra{\d_\h F(0)}^{-1} (z)} u \d \cS_\dist^2 \d \cL^2 (z) = J_\h F(0),
\end{split}\end{equation*}
i.e.\ ~\eqref{eq:functional-density} is proven, hence the thesis.
\end{proof}

\printbibliography    
\end{document}